\documentclass[preprint,12pt,times]{elsarticle}

\usepackage{amssymb}
\usepackage{amsmath}
\usepackage{amsthm}

\newtheorem{theorem}{Theorem}[section]
\newtheorem{corollary}[theorem]{Corollary}
\newtheorem{lemma}[theorem]{Lemma}
\newtheorem{proposition}[theorem]{Proposition}
\theoremstyle{definition}

\theoremstyle{remark}
\newtheorem{remark}[theorem]{Remark}

\numberwithin{equation}{section}
\usepackage[colorlinks, linkcolor=blue, anchorcolor=blue, citecolor=blue]{hyperref}
\usepackage{color}

\begin{document}

\begin{frontmatter}



\title{Difference of weighted composition operators between some spaces of analytic functions} 

\author{Jiaoye Du\fnref{label1}}
\ead{djy18632952319@163.com}

\author{Cezhong Tong\fnref{label1}}
\ead{ctong@hebut.edu.cn; cezhongtong@hotmail.com}


\author{Zicong Yang\fnref{label1}} 
\ead{zc25@hebut.edu.cn; zicongyang@126.com}

\affiliation[label1]{Department of Mathematics, Hebei University of Technology, Tianjin 300401, China}

\begin{abstract}
We first abtain a new and simpler proof of the main result in [IEOT, \textbf{93} (2021), 17], which characterized the bounded and compact differences of two weighted composition operators $C_{u,\varphi}-C_{v,\psi}$ acting between different Bergman spaces. More importantly, we get some characterizations for the difference of two weighted composition operators belonging to Schatten class. Futhermore, the compact difference of two weighted composition operators acting on Hardy-Hilbert spaces is also studied.
\end{abstract}

%

\begin{keyword}
Bergman space, Weighted composition operator, Difference, Schatten class.
\MSC[2020] 30H20; 47B33

\end{keyword}

\end{frontmatter}


\section{Introduction}

Let $\mathbb{D}$ be the open unit disk in the complex plane $\mathbb{C}$ and $H(\mathbb{D})$ be the space of all analytic functions on $\mathbb{D}$. For $\alpha>-1$ and $0<p<\infty$, the standard weighted Bergman space $A_{\alpha}^p(\mathbb{D})$ consists of all functions $f\in H(\mathbb{D})$ such that
\[\|f\|_{A_{\alpha}^p}=\left(\int_{\mathbb{D}}|f(z)|^p{\rm d}A_{\alpha}(z)\right)^{1/p}\]
is finite. Here ${\rm d}A_{\alpha}(z)=(\alpha+1)(1-|z|^2)^{\alpha}{\rm d}A(z)$ and ${\rm d}A=\frac{1}{\pi}{\rm d}x{\rm d}y$ is the normalized Lebesgue measure on $\mathbb{D}$. It is known that $A_{\alpha}^p(\mathbb{D})$ is a Banach space for $1\leq p<\infty$. In particular, $A_{\alpha}^2(\mathbb{D})$ is a Hilbert space with the following inner product
\[\langle f, g\rangle=\int_{\mathbb{D}}f(z)\overline{g(z)}{\rm d}A_{\alpha}(z),\quad f,g\in A_{\alpha}^2(\mathbb{D}).\]
When $0<p<1$, the space $A_{\alpha}^p(\mathbb{D})$ is a complete metric space under the translation invariant metric $(f,g)\mapsto \|f-g\|_{A_{\alpha}^p}^p$.

The Hardy space $H^p(\mathbb{D})$ is defined by 
\[H^p(\mathbb{D})=\left\{f\in H(\mathbb{D}): \|f\|_{H^p}^p=\sup_{0<r<1}\int_{0}^{2\pi}|f(re^{i\theta})|^p\frac{{\rm d}\theta}{2\pi}<\infty\right\}.\]
Hardy space can be viewed as the limiting space of $A_{\alpha}^p(\mathbb{D})$ as $\alpha\to -1^+$. And the well-known Littlewood-Paley identity asserts that the $H^2$-norm can be converted to an area integral:
\begin{equation}\label{equa1.1}
\|f\|_{H^2}^2\simeq |f(0)|^2+\int_{\mathbb{D}}|f'(z)|^2(1-|z|^2){\rm d}A(z).
\end{equation}

The theory of analytic function spaces and operators acting on them, especially weighted composition operators, has witnessed an explosive growth in the past several decades. Denoted by $S(\mathbb{D})$ the set of all analytic self-maps of $\mathbb{D}$. Given $u\in H(\mathbb{D})$ and $\varphi\in S(\mathbb{D})$, the weighted composition operator $C_{u,\varphi}$ on $H(\mathbb{D})$ is defined by
\[C_{u,\varphi}f=u\cdot f\circ\varphi.\]
It is known that weighted composition operators are closely related to the isometries on Hardy and Bergman spaces, see for example \cite{13,14}. When $u=1$, it reduces to the composition operator $C_{\varphi}$. One can refer to the standard references \cite{9} and \cite{23} for various aspects on the theory of composition operators and refer to \cite{10,11} for the study of weighted composition operators on weighted Bergman spaces.

In the course of research on the topological structure of the space of all composition operators, which was initiated by Berkson \cite{2} and then by Shapiro and Sundberg \cite{24} in more details, problems of characterizing the path components and compact differences of two composition operators attracted broad interest. On Bergman spaces, the effort to characterize the compact difference of two composition operators was initiated by Moorhouse \cite{19} on $A_{\alpha}^2(\mathbb{D})$ and then by Saukko \cite{21, 22} from $A_{\alpha}^p(\mathbb{D})$ to $A_{\alpha}^q(\mathbb{D})$. By using Joint-Carleson measures, Koo and Wang \cite{15} studied the bounded and compact difference $C_{\varphi}-C_{\psi}$ on $A_{\alpha}^p$ over the unit ball. For more results about the difference of composition operators on various settings, see \cite{7,8,26} and references therein. 

Acharyya and Wu \cite{1} first considered compact differences of two weighted composition operators $C_{u,\varphi}-C_{v,\psi}$ between weighted Bergman spaces with the weights $u,v$ satisfying a certain growth condition. Recently, Choe et al. \cite{4} obtained complete characterizations for bounded and compact differences $C_{u,\varphi}-C_{v,\psi}$ on $A_{\alpha}^p(\mathbb{D})$ under an extra condition: $u,v\in A_{\alpha}^p(\mathbb{D})$. And then, they extended their results to operators acting from $A_{\alpha}^p(\mathbb{D})$ to $A_{\alpha}^q(\mathbb{D})$ under a condition that $u,v\in A_{\alpha}^q(\mathbb{D})$, see \cite{5}.

We first state the results in \cite{5}. To this end, we introduce several notations. For $\varphi,\psi\in S(\mathbb{D})$, put
\[\rho(z)=d(\varphi(z),\psi(z)),\quad z\in\mathbb{D},\]
where $d$ denotes the pseudo-hyperbolic distance on $\mathbb{D}$, see Section 2.1. Given a positive Borel measure $\mu$ on $\mathbb{D}$ and $\varphi\in S(\mathbb{D})$, the pull-back measure $\mu\circ\varphi^{-1}$ is defined by 
\[(\mu\circ\varphi)^{-1}(E)=\mu(\varphi^{-1}(E))\]
for any Borel set $E\subset\mathbb{D}$. For $u,v\in H(\mathbb{D})$, $\varphi,\psi\in S(\mathbb{D})$ and $\beta>0$, we now define several pull-back measures as follows:
\begin{align*}
\omega_{\varphi,u}^q:=\left(|\rho u|^q{\rm d}A_{\alpha}\right)\circ\varphi^{-1},\quad \omega_{\psi,v}^q:=\left(|\rho v|^q{\rm d}A_{\alpha}\right)\circ\psi^{-1},
\end{align*}
and
\begin{align*}
\sigma_{\varphi,\beta}^q:=\left[(1-\rho)^{\beta}|u-v|^q{\rm d}A_{\alpha}\right]\circ\varphi^{-1},\quad \sigma_{\psi,\beta}^q:=\left[(1-\rho)^{\beta}|u-v|^q{\rm d}A_{\alpha}\right]\circ\psi^{-1}.
\end{align*}
For simplicity, put
\[\omega_q=\omega_{\varphi,u}^q+\omega_{\psi,v}^q\quad {\rm and }\quad \sigma_{\beta,q}=\sigma_{\varphi,\beta}^q+\sigma_{\psi,\beta}^q.\]
The results in \cite{5} are described in terms of $(p,q)$-Bergman Carleson measures, see Section 2.3 for details about this notion.

\begin{theorem}[\cite{5}]\label{theorem1.1}
Let $\alpha>-1$, $0<p\leq q<\infty$ and $\beta>\frac{q}{p}(\alpha+2)$. Suppose $u,v\in A_{\alpha}^q(\mathbb{D})$ and $\varphi,\psi\in S(\mathbb{D})$. Then the following statements are equivalent.
\begin{itemize}
\item[(i)] $C_{u,\varphi}-C_{v,\psi}: A_{\alpha}^p(\mathbb{D})\to A_{\alpha}^q(\mathbb{D})$ is bounded (compact, resp.).
\item[(ii)] $\omega_q$ and $\sigma_{\beta,q}$ are (vanishing, resp.) $(p,q)$-Bergman Carleson measures.  
\end{itemize}
\end{theorem}

\begin{theorem}[\cite{5}]\label{theorem1.2}
Let $\alpha>-1$, $0<q<p<\infty$ and $\beta>\frac{q}{p}(\alpha+1+\max\{1,p\})$. Suppose $u,v\in A_{\alpha}^q(\mathbb{D})$ and $\varphi,\psi\in S(\mathbb{D})$. Then the following statements are equivalent.
\begin{itemize}
\item[(i)] $C_{u,\varphi}-C_{v,\psi}:A_{\alpha}^p(\mathbb{D})\to A_{\alpha}^q(\mathbb{D})$ is bounded;
\item[(ii)] $C_{u,\varphi}-C_{v,\psi}: A_{\alpha}^p(\mathbb{D})\to A_{\alpha}^q(\mathbb{D})$ is compact;
\item[(iii)] $\omega_q$ and $\sigma_{\beta,q}$ are $(p,q)$-Bergman Carleson measures.
\end{itemize}
\end{theorem}

In order to prove $(i)\Rightarrow(ii)$ in Theorem \ref{theorem1.1} and $(i)\Rightarrow(iii)$ in Theorem \ref{theorem1.2}, which are the hardest steps, several technical lemmas were applied in \cite{5}, see also \cite{4} for the case $p=q$. To the best of our knowledge, most of the studies of the difference of weighted composition operators rely on these lemmas in \cite{4} and \cite{5}, see for example \cite{16, 20}. Our first aim in this paper is to construct a new method and give a simpler proof for $(i)\Rightarrow (ii)$ in Theorem \ref{theorem1.1} and $(i)\Rightarrow(iii)$ in Theorem \ref{theorem1.2}. Moreover, the condition $u,v\in A_{\alpha}^q(\mathbb{D})$ is not needed in our proof. And the proof does not rely on those technical lemmas. Quite recently, Choe et al. \cite{3} extended the results in \cite{5} to the ball setting and removed the $L^q$-condition of the weights $u,v$. We point out that our method here is different from \cite{3}.

In addition to studying the bounded and compact differences of (weighted) composition operators, Choe   et al. \cite{6} characterized the Hilbert-Schmidt property of the difference $C_{\varphi}-C_{\psi} $ on $A_{\alpha}^2(\mathbb{D})$ and showed that $C_{\varphi}$ and $C_{\psi}$ are in the same path component in the space of composition operators under the Hilbert-Schmidt norm if and only if $C_{\varphi}-C_{\psi}$ is Hilbert-Schmidt. The result was then extended to the ball setting by Zhang and Zhou \cite{27}. Acharyya and Wu \cite{1} first characterized Hilbert-Schmidt difference of two weighted composition operators acting from Bergman space or Hardy space to an $L^2(\mu)$ space. For $0<p<\infty$, we recall that a compact operator $T$ acting on a separated Hilbert space $\mathcal{H}$ belongs to the Schatten class $S_p(\mathcal{H})$ if its sequence of singular numbers belongs to the sequence space $l^p$ (the singular numbers are the square roots of the eigenvalues of the positive operator $T^{*}T$). $T$ is said to be Hilbert-Schmidt if $T\in S_2(\mathcal{H})$. One can refer to \cite[Chapter 1]{29} for some details about Schatten classes. As far as we know, there are few papers studying the Schatten-$p$ class differences of weighted composition operators. So our second aim in this paper is to give some characterizations of operator differences belonging to Schatten-$p$ class.

For a positive Borel measure $\mu$ on $\mathbb{D}$ and $r>0$, we put
\[M_r(\mu)(z)=\frac{\mu(D(z,r))}{(1-|z|^2)^{2+\alpha}},\quad z\in\mathbb{D}\]
for the averaging function of $\mu$, where $D(z,r)$ is the Bergman metric disk centered at $z$ with radius $r$, see Section 2.1 and 2.4. Let 
\[{\rm d}\lambda(z)=\frac{{\rm d}A(z)}{(1-|z|^2)^2}\]
be the M$\ddot{\rm o}$bius invariant measure on $\mathbb{D}$. Our main result states as follows.

\begin{theorem}\label{theorem1.3}
Let $\alpha>-1$, $0<p<\infty$ and $\beta>0$. Suppose $u,v\in H(\mathbb{D})$ and $\varphi,\psi\in S(\mathbb{D})$. If $M_r(\omega_2)$ and $M_r(\sigma_{\beta,2})$ belong to $L^{\frac{p}{2}}(\mathbb{D},d\lambda)$ for some (or any) $r>0$, then $C_{u,\varphi}-C_{v,\psi}\in S_p(A_{\alpha}^2)$. Moreover, when $p\geq 2$ and $\beta\geq 2(\alpha+2)$, the above condition is also necessary for $C_{u,\varphi}-C_{v,\psi}\in S_p(A_{\alpha}^2)$.
\end{theorem}

By \eqref{equa1.1}, we know that $f\in H^2(\mathbb{D})$ if and only if $f'\in A_{1}^2(\mathbb{D})$. Choe et al. \cite{4} showed that $C_{\varphi}-C_{\psi}$ is compact on $H^2(\mathbb{D})$ if and only if $C_{\varphi',\varphi}-C_{\psi',\psi}$ is compact on $A_1^2(\mathbb{D})$ and obtained a compact criterion for the difference of two composition operators acting on $H^2(\mathbb{D})$. However, it will be more difficult to characterize the bounded and compact differences of two weighted composition operators $C_{u,\varphi}-C_{v,\psi}$ on Hardy spaces. Our next result shows a rigidity of the difference $C_{u,\varphi}-C_{v,\psi}$ on $H^2(\mathbb{D})$.
 
\begin{theorem}\label{theorem1.4}
Let $u,v\in H(\mathbb{D})$ and $\varphi,\psi\in S(\mathbb{D})$. Suppose $C_{u,\varphi}$ and $C_{v,\psi}$ are bounded on $H^2(\mathbb{D})$, then $C_{u,\varphi}-C_{v,\psi}$ is compact on $H^2(\mathbb{D})$ if and only if both $C_{u',\varphi}-C_{v',\psi}: H^2(\mathbb{D})\to A_1^2(\mathbb{D})$ and $C_{u\varphi',\varphi}-C_{v\psi',\psi}:A_1^2(\mathbb{D})\to A_1^2(\mathbb{D})$ are compact.
\end{theorem}

This paper is organized as follows. In Section 2, we collect some preliminary facts and auxiliary lemmas that will be used later. Section 3 is devoted to the new simpler proof of $(i)\Rightarrow(ii)$ in Theorem \ref{theorem1.1} and $(i)\Rightarrow(iii)$ in Theorem \ref{theorem1.2}. The proof of Theorem \ref{theorem1.3} is also presented in Section 4. Moreover, we obtain a Schatten-$p$ class characterization for the difference $C_{u,\varphi}-C_{v,\psi}$ via Reproducing Kernel Thesis. In Section 5, we prove Theorem \ref{theorem1.4} and characterize the linear sums of two composition operators on $H^2(\mathbb{D})$.

Throughout the paper we use the same letter $C$ to denote positive constants which may vary at different occurrences but do not depend on the essential argument. For non-negative quantities $A$ and $B$, we write $A\lesssim B$ (or equivalently $B\gtrsim A$) if there exists an absolute constant $C>0$ such that $A\leq CB$. $A\simeq B$ means both $A\lesssim B$ and $B\lesssim A$.

\section{Preliminaries}

In this section, we recall some basic facts and present some auxiliary lemmas that will be used in the sequel.

\subsection{Pseudo-hyperbolic Distance}

The pseudo-hyperbolic distance between $z,w\in\mathbb{D}$ is given by 
\begin{equation*}
d(z,w)=\left|\frac{z-w}{1-\overline{z}w}\right|.
\end{equation*}
And the Bergman metric between $z,w\in\mathbb{D}$ is given by 
\begin{equation*}
\beta(z,w)=\frac{1}{2}\log\frac{1+d(z,w)}{1-d(z,w)}.
\end{equation*}
Write $D(z,r)=\{w\in\mathbb{D}: \beta(z,w)<r\}$ for the Bergman metric disk centered at $z$ with radius $r>0$.

It is easy to check that
\begin{equation}\label{equa2.1}
1-d^2(z,w)=\frac{(1-|z|^2)(1-|w|^2)}{|1-\overline{z}w|^2}.
\end{equation}

Given $r>0$, it is also well-known that
\begin{equation}\label{equa2.2}
A_{\alpha}(D(z,r))\simeq (1-|z|^2)^{2+\alpha},
\end{equation}
and 
\begin{equation}\label{equa2.3}
1-|z|^2\simeq 1-|w|^2\simeq |1-\overline{z}w|
\end{equation}
for all $z\in\mathbb{D}$ and $w\in D(z,r)$. Moreover,
\begin{equation}\label{equa2.4}
|1-\overline{a}z|\simeq |1-\overline{a}w|
\end{equation}
for all $a$, $z$ and $w$ in $\mathbb{D}$ with $\beta(z,w)<r$. Here, all the constants suppressed in these estimates depend only on $r$.

The following lemma, which is used in the proof of Theorem \ref{theorem1.3}, can be found in \cite[Lemma 2.2]{15}.

\begin{lemma}[\cite{15}]\label{lemma2.1}
Let $\alpha>-1$, $0<p<\infty$ and $0<r_1<r_2$. Then there is a constant $C=C(\alpha, p, r_1, r_2)$ such that
\begin{equation*}
|f(z)-f(w)|^p\leq C\frac{d^p(z,w)}{(1-|z|^2)^{2+\alpha}}\int_{D(z,r_2)}|f(\zeta)|^p{\rm d}A_{\alpha}(\zeta)
\end{equation*}
for any $f\in A_{\alpha}^p(\mathbb{D})$ and $z,w\in\mathbb{D}$ with $\beta(z,w)<r_1$.
\end{lemma}

A sequence $\{a_j\}$ in $\mathbb{D}$ is called an $r$-lattice in the Bergman metric if the following conditions hold:
\begin{itemize}
\item[(i)] $\bigcup\limits_{j=1}^{\infty}D(a_j,r)=\mathbb{D}$,
\item[(ii)] $\{D(a_j,r/2)\}_{j=1}^{\infty}$ are pairwise disjoint.
\end{itemize}
For any $r>0$, the existence of $r$-lattice in the Bergman metric is ensured by \cite[Lemma 4.8]{29}. And according to \cite[Lemma 4.7]{29}, with hypotheses (i) and (ii), it is easy to check that
\begin{itemize}
\item[(iii)] there exists a positive integer $N_0$ such that every point in $\mathbb{D}$ belongs to at most $N_0$ of the sets $D(a_j,4r)$. 
\end{itemize}

\subsection{Test Functions}

Recall the following pointwise estimate for derivatives of functions in $A_{\alpha}^p(\mathbb{D})$, see for example \cite[Lemma 2.1]{17}.

\begin{lemma}\label{lemma2.2}
Let $\alpha>-1$, $0<p<\infty$, $r>0$ and $n$ be an non-negative integer. There exists a constant $C=C(\alpha, p,r,n)>0$ such that
\begin{equation*}
|f^{(n)}(z)|^p\leq \frac{C}{(1-|z|^2)^{2+\alpha+np}}\int_{D(z,r)}|f(w)|^p{\rm d}A_{\alpha}(w)
\end{equation*}
for all $f\in A_{\alpha}^p(\mathbb{D})$ and $z\in\mathbb{D}$.
\end{lemma}

Note from Lemma \ref{lemma2.2} that each point evaluation of $n$th order: $f\mapsto f^{(n)}(z)$, is a continuous linear functional on $A_{\alpha}^p(\mathbb{D})$ for every $z\in\mathbb{D}$. Thus when $p=2$, by the Riesz representation theorem in Hilbert space theory, to each $z\in\mathbb{D}$ corresponds a unique function $K_z^{[n]}\in A_{\alpha}^2(\mathbb{D})$ such that
\begin{equation}\label{equa2.5}
f^{(n)}(z)=\langle f,K_z^{[n]}\rangle, \quad \forall f\in A_{\alpha}^2(\mathbb{D}).
\end{equation} 
$K_z^{[n]}$ is called the reproducing kernel function in $A_{\alpha}^2(\mathbb{D})$ at $z$ of order $n$, whose explicit formula is known as 
\begin{equation*}
K_z(w)=K_z^{[0]}(w)=\frac{1}{(1-\overline{z}w)^{2+\alpha}},\quad w\in\mathbb{D},
\end{equation*}
and $K_z^{[n]}(w)=\frac{\partial ^nK_z}{\partial\overline{z}^n}(w)$. It is known that
\begin{equation}\label{equa2.6}
\|K_z^{[n]}\|_{A_{\alpha}^2}^2=(K_z^{[n]})^{(n)}(z)\simeq \frac{1}{(1-|z|^2)^{2+\alpha+2n}}.
\end{equation}
Let $N,i$ be non-negative integers such that $N>\frac{2+\alpha}{p}$. For any $a\in\mathbb{D}$, let
\begin{equation*}
F_{a,N}^{[i]}(z)=\frac{z^i}{(1-\overline{a}z)^{N+i}},\quad z\in\mathbb{D}.
\end{equation*} 
According to \cite[Lemma 4.26]{29} and Forelli-Rudin's estimate, see \cite[Theorem 1.12]{28}, we have
\begin{equation*}
\|F_{a,N}^{[i]}\|_{A_{\alpha}^p}\simeq (1-|a|^2)^{\frac{2+\alpha}{p}-N-i}\quad {\rm and}\quad  \|F_{a,N}^{[i]}\|_{H^p}\simeq (1-|a|^2)^{\frac{1}{p}-N-i}.
\end{equation*}
Constants suppressed above are independent of $a$. We set 
\begin{equation*}
f_{a,N,p}^{[i]}(z)=(1-|a|^2)^{N+i-\frac{2+\alpha}{p}}F_{a,N}^{[i]}(z)
\end{equation*}
and
\begin{equation*}
g_{a,N,p}^{[i]}(z)=(1-|a|^2)^{N+i-\frac{1}{p}}F_{a,N}^{[i]}(z).
\end{equation*}
Then $\|f_{a,N,p}^{[i]}\|_{A_{\alpha}^p}\simeq 1$, $\|g_{a,N,p}^{[i]}\|_{H^p}\simeq 1$ and both $f_{a,N,p}^{[i]}$ and $g_{a,N,p}^{[i]}$ converges to 0 uniformly on compact subsets of $\mathbb{D}$ as $|a|\to 1^-$.

Besides, we have the following result, which derives from \cite[Theorem 4.33]{29}.

\begin{lemma}\label{lemma2.3}
Let $\{a_j\}$ be an $r$-lattice in the Bergman metric. Suppose $\alpha>-1$, $0<p<\infty$ and $N,i$ be non-negative integers with $N>\max\{1,\frac{1}{p}\}+\frac{1+\alpha}{p}$. For any $\mathbf{c}=\{c_j\}\in l^p$, let
\begin{equation*}
H_{\mathbf{c}, N,p}^{[i]}(z)=\sum_{j=1}^{\infty}c_j\frac{(1-|a_j|^2)^{N+i-\frac{2+\alpha}{p}}z^i}{(1-\overline{a}_jz)^{N+i}},\quad z\in\mathbb{D}.
\end{equation*}
Then $H_{\mathbf{c},N,p}^{[i]}\in A_{\alpha}^p(\mathbb{D})$ and $\|H_{\mathbf{c}, N,p}^{[i]}\|_{A_{\alpha}^p}\lesssim \|\{c_j\}\|_{l^p}$.
\end{lemma}

\subsection{Compact operator and Schatten-$p$ class}

Given two complete metrizable topological vector spaces $X$ and $Y$, a bounded linear operator $T: X\to Y$ is said to be compact if the image of each bounded sequence in $X$ has a convergent subsequence in $Y$.

The following compactness criterion derives from \cite[Lemma 2.1]{5}, see also \cite[Proposition 3.11]{9}.

\begin{lemma}\label{lemma2.4}
Let $\alpha\geq -1$ and $0<p,q<\infty$. Let $T$ be a linear combination of weighted composition operators and assume $T: A_{\alpha}^p(\mathbb{D})\to A_{\alpha}^q(\mathbb{D})$ is bounded. Then $T: A_{\alpha}^p(\mathbb{D})\to A_{\alpha}^q(\mathbb{D})$ is compact if and only if $\|Tf_n\|_{A_{\alpha}^q}\to 0$ for any bounded sequence in $A_{\alpha}^p(\mathbb{D})$ that converges to 0 uniformly on compact subsets of $\mathbb{D}$.
\end{lemma}

For any $z\in\mathbb{D}$, let $k_z^{[n]}=K_z^{[n]}/\|K_z^{[n]}\|_{A_{\alpha}^2}$. Suppose $T$ is a bounded linear operator on $A_{\alpha}^2$, we define a function 
\begin{equation*}
\widetilde{T}_{[n]}(z)=\langle Tk_z^{[n]}, k_z^{[n]}\rangle,\quad z\in\mathbb{D}.
\end{equation*}
$\widetilde{T}_{[n]}$ is called the $n$th-Berezin transform of $T$. When $n=0$, it reduces to the classcial Berezin transform.

\begin{lemma}\label{lemma2.5}
Suppose $T$ is a positive operator on $A_{\alpha}^2(\mathbb{D})$. If $T$ belongs to the trace class, i.e. $T\in S_1(A_{\alpha}^2)$, then $\widetilde{T}_{[n]}\in L^1(\mathbb{D},{\rm d}\lambda)$.
\end{lemma}

\begin{proof}
Fix an orthonormal basis $\{e_j\}$ for $A_{\alpha}^2(\mathbb{D})$. Since $T$ is positive, it belongs to the trace class if and only if
$$\sum_{j=1}^{\infty}\langle Te_j,e_j\rangle<\infty.$$
On the other hand, let $S=\sqrt{T}$. By \cite[Theorem 4.28]{29}, the reproducing formula \eqref{equa2.5} and \eqref{equa2.6}, we have
\begin{align*}
\|T\|_{S_1}&=\sum_{j=1}^{\infty}\langle Te_j,e_j\rangle=\sum_{j=1}^{\infty}\|Se_j\|_{A_{\alpha}^2}^2\\
&\gtrsim \sum_{j=1}^{\infty}\int_{\mathbb{D}}|(Se_j)^{(n)}(z)|^2(1-|z|^2)^{2n}{\rm d}A_{\alpha}(z)\\
&=\int_{\mathbb{D}}\left(\sum_{j=1}^{\infty}\langle Se_j, K_z^{[n]} \rangle^2\right)(1-|z|^2)^{2n}{\rm d}A_{\alpha}(z)\\
&=\int_{\mathbb{D}}\left(\sum_{j=1}^{\infty}\langle e_j, SK_z^{[n]}\rangle ^2\right)(1-|z|^2)^{2n}{\rm d}A_{\alpha}(z)\\
&=\int_{\mathbb{D}}\|SK_z^{[n]}\|_{A_{\alpha}^2}^2(1-|z|^2)^{2n}{\rm d}A_{\alpha}(z)\\
&=\int_{\mathbb{D}}\langle TK_z^{[n]}, K_z^{[n]}\rangle (1-|z|^2)^{2n}{\rm d}A_{\alpha}(z)\\
&\simeq \int_{\mathbb{D}}\langle Tk_z^{[n]}, k_z^{[n]}\rangle {\rm d}\lambda(z)=\int_{\mathbb{D}}\widetilde{T}_{[n]}(z){\rm d}{\lambda}(z).
\end{align*}
Therefore, $\int_{\mathbb{D}}\widetilde{T}_{[n]}(z){\rm d}\lambda(z)<\infty$. The proof is complete.
\end{proof}

\begin{lemma}\label{lemma2.6}
Suppose $T$ is a positive operator on $A_{\alpha}^2(\mathbb{D})$ and $1\leq p< \infty$. If $T\in S_p(A_{\alpha}^2)$, then $\widetilde{T}_{[n]}\in L^p(\mathbb{D},{\rm d}\lambda)$.
\end{lemma}

\begin{proof}
For $1\leq p<\infty$, by \cite[Proposition 1.31]{29}, we have
\[(\widetilde{T^p})_{[n]}(z)\geq \left(\widetilde{T}_{[n]}(z)\right)^p,\quad z\in\mathbb{D}.\]
Since $T$ is positive, $T\in S_p(A_{\alpha}^2)$ implies that $T^p$ is in the trace class. So it follows from Lemma \ref{lemma2.5} that
\begin{equation*}
\int_{\mathbb{D}}\left(\widetilde{T}_{[n]}(z)\right)^p{\rm d}\lambda(z)\leq\int_{\mathbb{D}}(\widetilde{T^p})_{[n]}(z){\rm d}\lambda(z)\lesssim \|T^p\|_{S_1}=\|T\|_{S_p}<\infty.
\end{equation*}
The proof is complete.
\end{proof}

\subsection{Carleson Measure}

Let $\mu$ be a positive Borel measure on $\mathbb{D}$. For $\alpha>-1$ and $0<p,q<\infty$, we say that $\mu$ is a $(p,q)$-Bergman Carleson measure if there exsits a constant $C>0$ such that
\begin{equation*}
\left(\int_{\mathbb{D}}|f(z)|^q{\rm d}\mu(z)\right)^{1/q}\leq C\|f\|_{A_{\alpha}^p}
\end{equation*}
for all $f\in A_{\alpha}^p(\mathbb{D})$. That is, $\mu$ is a $(p,q)$-Bergman Carleson measure if the embedding $A_{\alpha}^p(\mathbb{D})\subset L^q(\mathbb{D}, {\rm d}\mu)$ is continuous. If, in addition, the embedding $A_{\alpha}^p(\mathbb{D})\subset L^q(\mathbb{D},{\rm d}\mu)$ is compact, then $\mu$ is called a vanishing $(p,q)$-Bergman Carleson measure. Note that $(p,q)$-Bergman Carleson measures are finite.

For $r>0$ and $t>0$, we put
\[M_{r,t}(\mu)(z)=\frac{\mu(D(z,r))}{(1-|z|^2)^{t(2+\alpha)}},\quad z\in\mathbb{D}\]
for the weighted averaging function of $\mu$. When $t=1$, write $M_r(\mu):=M_{r,1}(\mu)$ for simplicity.

The geometric characterizations for $(p,q)$-Bergman Carleson measures have been known for some time, see \cite{18} and we state the results as follows.

\begin{lemma}\label{lemma2.7}
Let $0<p,q<\infty$ and $\mu$ be a positive Borel measure on $\mathbb{D}$.
\begin{itemize}
\item[(i)] If $0<p\leq q<\infty$, then $\mu$ is a $(p,q)$-Bergman Carleson measure if and only if 
$$\sup_{z\in\mathbb{D}}M_{r,\frac{q}{p}}(z)<\infty$$
for some (or any) $r>0$, and $\mu$ is a vanishing $(p,q)$-Bergman Carleson measure if and only if
$$\lim_{|z|\to 1^-}M_{r,\frac{q}{p}}(z)=0$$
for some (or any) $r>0$.
\item[(ii)] If $0<q<p<\infty$, then $\mu$ is a $(p,q)$-Bergman Carleson measure if and only if it is a vanishing $(p,q)$-Bergman Carleson measure if and only if
$$M_r(\mu)\in L^{\frac{p}{p-q}}(\mathbb{D},{\rm d}A_{\alpha})$$
for some (or any) $r>0$.
\end{itemize}
\end{lemma}

When $p=q$, we just call $\mu$ is a (vanishing) Bergman Carleson measure for simplicity, since the characterizations for $(p,p)$-Bergman Carleson measures are independent of $p$.

When $\mu$ is a Bergman Carleson measure and $r>0$, by Lemma \ref{lemma2.2} and Fubini's Theorem, it is easy to see that
\begin{equation}\label{equa2.7}
\int_{\mathbb{D}}|f(z)|^p{\rm d}\mu(z)\lesssim \int_{\mathbb{D}}|f(w)|^pM_r(\mu)(w){\rm d}A_{\alpha}(w)
\end{equation} 
for $0<p<\infty$ and any $f\in H(\mathbb{D})$. 

Given a positive Borel measure $\mu$ on $\mathbb{D}$, the Toeplitz operator $T_{\mu}$ acting on $A_{\alpha}^2(\mathbb{D})$ is the densely defined integral operator
\[T_{\mu}f(z)=\int_{\mathbb{D}}\frac{f(w)}{(1-z\overline{w})^{2+\alpha}}{\rm d}\mu(w),\quad z\in\mathbb{D}.\]
Clearly, $T_{\mu}$ is well defined on $H^{\infty}(\mathbb{D})$, the space of all bounded analytic functions on $\mathbb{D}$.

It is known that $T_{\mu}$ is bounded (compact, resp.) on $A_{\alpha}^2(\mathbb{D})$ if and only if $\mu$ is a (vanishing, resp.) Bergman Carleson measure. And for $0<p<\infty$, according to \cite[Theorem 7.18]{29}, $T_{\mu}$ is in $S_p(A_{\alpha}^2)$ if and only if 
\begin{equation}\label{equa2.8}
M_r(\mu)\in L^p(\mathbb{D},{\rm d}\lambda)
\end{equation}
for some (or any) $r>0$.

\begin{lemma}\label{lemma2.8}
Suppose $\mu$ is a vanishing Bergman Carleson measure and $R>r>0$. Let ${\rm d}\nu=M_r(\mu){\rm d}A_{\alpha}$. If $M_R(\mu)\in L^p(\mathbb{D},{\rm d}\lambda)$, then $T_\nu$ is in $S_p(A_{\alpha}^2)$.
\end{lemma}

\begin{proof}
By \eqref{equa2.3}, it is easy to see that $M_{R-r}(\nu)\lesssim M_R(\mu)$. So $M_{R-r}(\nu)\in L^{p}(\mathbb{D},{\rm d}\lambda)$, and then $T_{\nu}\in S_p(A_{\alpha}^2)$.
\end{proof}

\section{New proof of Theorems 1.1 and 1.2}

In this section, we present a new and simpler proof of $(i)\Rightarrow(ii)$ in Theorem \ref{theorem1.1} and $(i)\Rightarrow(iii)$ in Theorem \ref{theorem1.2}, which characterize the necessity for the boundedness and compactness of the difference $C_{u,\varphi}-C_{v,\psi}$ from $A_{\alpha}^p(\mathbb{D})$ to $A_{\alpha}^q(\mathbb{D})$.

\begin{proof}[{\bf New Proof of $(i)\Rightarrow(ii)$ in Theorem 1.1}]
Assume $p\leq q$ and $C_{u,\varphi}-C_{v,\psi}$ is bounded from $A_{\alpha}^p(\mathbb{D})$ to $A_{\alpha}^q(\mathbb{D})$. Then
\begin{equation*}
\sup_{a\in\mathbb{D}}\|(C_{u,\varphi}-C_{v,\psi})f_{a,N,p}^{[i]}\|_{A_{\alpha}^q}<\infty, \quad i=0,1.
\end{equation*}
Fix $r>0$. By \eqref{equa2.3}, we get
\begin{align*}
&\|(C_{u,\varphi}-C_{v,\psi})f_{a,N,p}^{[0]}\|_{A_{\alpha}^q}^q\\
&\quad \geq \int_{\varphi^{-1}(D(a,r))}\left|\frac{u(z)}{(1-\overline{a}\varphi(z))^N}-\frac{v(z)}{(1-\overline{a}\psi(z))^N}\right|^q(1-|a|^2)^{(N-\frac{2+\alpha}{p})q}{\rm d}A_{\alpha}(z)\\
&\quad \gtrsim\int_{\varphi^{-1}(D(a,r))}\left|u(z)-v(z)\left(\frac{1-\overline{a}\varphi(z)}{1-\overline{a}\psi(z)}\right)^N\right|^q\frac{1}{(1-|a|^2)^{\frac{q}{p}(2+\alpha)}}{\rm d}A_{\alpha}(z),
\end{align*}
and similarly,
\begin{align*}
&\|(C_{u,\varphi}-C_{v,\psi})f_{a,N,p}^{[1]}\|_{A_{\alpha}^q}^q\\
&\quad \gtrsim \frac{\int_{\varphi^{-1}(D(a,r))}\left|u(z)\varphi(z)-v(z)\psi(z)\left(\frac{1-\overline{a}\varphi(z)}{1-\overline{a}\psi(z)}\right)^{N+1}\right|^q{\rm d}A_{\alpha}(z)}{(1-|a|^2)^{\frac{q}{p}(\alpha+2)}}.
\end{align*}
Therefore,
\begin{equation}\label{equa3.1}
\sup_{a\in\mathbb{D}}\frac{\int_{\varphi^{-1}(D(a,r))}\left|u(z)-v(z)\left(\frac{1-\overline{a}\varphi(z)}{1-\overline{a}\psi(z)}\right)^N\right|^q{\rm d}A_{\alpha}(z)}{(1-|a|^2)^{\frac{q}{p}(\alpha+2)}}<\infty,
\end{equation}
and
\begin{equation}\label{equa3.2}
\sup_{a\in\mathbb{D}}\frac{\int_{\varphi^{-1}(D(a,r))}\left|u(z)\varphi(z)-v(z)\psi(z)\left(\frac{1-\overline{a}\varphi(z)}{1-\overline{a}\psi(z)}\right)^{N+1}\right|^q{\rm d}A_{\alpha}(z)}{(1-|a|^2)^{\frac{q}{p}(\alpha+2)}}<\infty.
\end{equation}
By \eqref{equa2.3}, $\left|\frac{1-\overline{a}\varphi(z)}{1-\overline{a}\psi(z)}\right|\lesssim 1$ when $\varphi(z)\in D(a,r)$. Multiplying the integrand in \eqref{equa3.1} by $|\psi(z)|^q\left|\frac{1-\overline{a}\varphi(z)}{1-\overline{a}\psi(z)}\right|^q$, adding it to \eqref{equa3.2}, and by the triangle inequality, we obtain
\begin{equation}\label{equa3.3}
\sup_{a\in\mathbb{D}}\frac{\int_{\varphi^{-1}(D(a,r))}|u(z)|^q\left|\frac{\varphi(z)-\psi(z)}{1-\overline{a}\psi(z)}\right|^q{\rm d}A_{\alpha}(z)}{(1-|a|^2)^{\frac{q}{p}(2+\alpha)}}<\infty.
\end{equation}
By \eqref{equa2.4}, $|1-\overline{a}\psi(z)|\simeq |1-\overline{\varphi(z)}\psi(z)|$ when $\varphi(z)\in D(a,r)$. So it follows from \eqref{equa3.3} that
\begin{align*}
\sup_{a\in\mathbb{D}}\frac{\int_{\varphi^{-1}(D(a,r))}|u(z)|^q\rho(z)^q{\rm d}A_{\alpha}(z)}{(1-|a|^2)^{\frac{q}{p}(\alpha+2)}}<\infty.
\end{align*}
This shows that $\omega_{\varphi,u}^q$ is a $(p,q)$-Bergman Carleson measure by Lemma \ref{lemma2.7}. Similarly, $\omega_{\psi,v}^q$ is also a $(p,q)$-Bergman Carleson measure.

On the other hand, when $\varphi(z)\in D(a,r)$, we have
\begin{equation}\label{equa3.4}
\begin{split}
&|u(z)-v(z)|\left|\frac{1-\overline{a}\varphi(z)}{1-\overline{a}\psi(z)}\right|^N\\
&\quad \leq \left|u(z)-v(z)\left(\frac{1-\overline{a}\varphi(z)}{1-\overline{a}\psi(z)}\right)^N\right|+|u(z)|\left|1-\left(\frac{1-\overline{a}\varphi(z)}{1-\overline{a}\psi(z)}\right)^N\right|\\
&\quad \lesssim \left|u(z)-v(z)\left(\frac{1-\overline{a}\varphi(z)}{1-\overline{a}\psi(z)}\right)^N\right|+|u(z)|\left|\frac{\varphi(z)-\psi(z)}{1-\overline{a}\psi(z)}\right|
\end{split}
\end{equation}
And by \eqref{equa2.1}, \eqref{equa2.3} and \eqref{equa2.4}, we have
\begin{equation}\label{equa3.5}
\left|\frac{1-\overline{a}\varphi(z)}{1-\overline{a}\psi(z)}\right|\gtrsim \frac{(1-|\varphi(z)|^2)(1-|\psi(z)|^2)}{|1-\overline{\varphi(z)}\psi(z)|^2}=1-\rho^2(z)
\end{equation}
for $z\in \varphi^{-1}(D(a,r))$. So we combine \eqref{equa3.1} and \eqref{equa3.3} to obtain
\begin{equation*}
\sup_{a\in\mathbb{D}}\frac{\int_{\varphi^{-1}(D(a,r))}|u(z)-v(z)|^q(1-\rho(z))^{qN}{\rm d}A_{\alpha}(z)}{(1-|a|^2)^{\frac{q}{p}(2+\alpha)}}<\infty.
\end{equation*}
This means that $\sigma_{\varphi,\beta}^q$ is a $(p,q)$-Bergman Carleson measure for $\beta\geq qN>\frac{q}{p}(2+\alpha)$ by Lemma \ref{lemma2.7}. Similarly, $\sigma_{\psi,\beta}^q$ is also a $(p,q)$-Bergman Carleson measure.

Finally, the proof for the compactness part is just a modification. By Lemma \ref{lemma2.4}, one only needs to write ``$\lim\limits_{|a|\to 1}$'' instead of ``$\sup\limits_{a\in\mathbb{D}}$'' and write ``$\to 0$'' instead of ``$<\infty$'', respectively. We omit the routine details.
\end{proof}

According to the above proof, we conclude the following corollary, which coincides with \cite[Theorem 1.1]{3}.

\begin{corollary}
Let $\alpha>-1$, $0<p\leq q<\infty$ and $N>\frac{2+\alpha}{p}$. Suppose $u,v\in H(\mathbb{D})$ and $\varphi,\psi\in S(\mathbb{D})$. Then
\begin{itemize}
\item[(i)] $C_{u,\varphi}-C_{v,\psi}: A_{\alpha}^p(\mathbb{D})\to A_{\alpha}^q(\mathbb{D})$ is bounded if and only if 
\begin{equation*}
\sup_{a\in\mathbb{D}}\|(C_{u,\varphi}-C_{v,\psi})f_{a,N,p}^{[i]}\|_{A_{\alpha}^q}<\infty, \quad i\in \mathbb{N}.
\end{equation*}
\item[(ii)] 
$C_{u,\varphi}-C_{v,\psi}: A_{\alpha}^p(\mathbb{D})\to A_{\alpha}^q(\mathbb{D})$ is compact if and only if 
\begin{equation*}
\lim_{|a|\to 1}\|(C_{u,\varphi}-C_{v,\psi})f_{a,N,p}^{[i]}\|_{A_{\alpha}^q}=0, \quad i\in \mathbb{N}.
\end{equation*}
\end{itemize}
\end{corollary}

Now we proceed to the proof of $(i)\Rightarrow (iii)$ in Theorem \ref{theorem1.2}. Before that, we recall some basic facts about the Khinchine's inequality, which is an important tool in complex and functional analysis. An introduction to this topic can be found in Appendix of \cite{12}.

Let $\{r_j(t)\}$ denotes the sequence of Rademacher functions defined by
$$r_0(t)=\begin{cases}1, &\mbox{if}\ 0\leq t-[t]<\frac{1}{2}\\-1, &\mbox{if}\ \frac{1}{2}\leq t-[t]<1,\end{cases}$$ 
where $[t]$ denotes the largest integer not greater than $t$ and $r_j(t)=r_0(2^jt)$ for $j=1, 2,\cdots$.
If $0<p<\infty$, then Khinchine's inequality states that
$$\left(\sum_{j}\left|c_{j}\right|^{2}\right)^{p / 2} \simeq \int_{0}^{1}\left|\sum_{j} c_{j} r_{j}(t)\right|^{p} {\rm d}t$$
for complex sequences $\left\{c_{j}\right\}$

\begin{proof}[{\bf New proof of $(i)\Rightarrow (iii)$ in Theorem 1.2}]
Assume $0<q<p<\infty$ and $C_{u,\varphi}-C_{v,\psi}: A_{\alpha}^p(\mathbb{D})\to A_{\alpha}^q(\mathbb{D})$ is bounded. Fix $r>0$. Let $\{a_j\}$ be an $r$-lattice in the Bergman metric and $\mathbf{c}=\{c_j\}\in l^p$. For $N>\max\{1,\frac{1}{p}\}+\frac{1+\alpha}{p}$, we use the test function $H_{\mathbf{c},N,p}^{[i]}$ defined in Lemma \ref{lemma2.3} to obtain
\begin{equation*}
\|(C_{u,\varphi}-C_{v,\psi})H_{\mathbf{c},N,p}^{[i]}\|_{A_{\alpha}^q}\lesssim \|\{c_j\}\|_{l^p}, \quad i=0,1.
\end{equation*}
For $i=0$, we obtain
\begin{equation}\label{equa3.6}
\begin{split}
&\int_{\mathbb{D}}\left|\sum_{j=1}^{\infty}c_j\left(u(z)\frac{(1-|a_j|^2)^{N-\frac{2+\alpha}{p}}}{(1-\overline{a}_j\varphi(z))^N}-v(z)\frac{(1-|a_j|^2)^{N-\frac{2+\alpha}{p}}}{(1-\overline{a}_j\psi(z))^N}\right)\right|^q{\rm d}A_{\alpha}(z)\\
&\quad \lesssim \|\{c_j\}\|_{l^p}^q.
\end{split}
\end{equation}
In \eqref{equa3.6}, we replace $c_j$ by $r_j(t)c_j$, so that the right hand side does not change. Then intergrating both sides with respect to $t$ from 0 to 1, and by Khinchine's inequaltiy and Fubini's Theorem, we obtain
\begin{equation*}
\begin{split}
&\int_{\mathbb{D}}\left(\sum_{j=1}^{\infty}|c_j|^2\left|u(z)\frac{(1-|a_j|^2)^{N-\frac{2+\alpha}{p}}}{(1-\overline{a}_j\varphi(z))^N}-v(z)\frac{(1-|a_j|^2)^{N-\frac{2+\alpha}{p}}}{(1-\overline{a}_j\psi(z))^N}\right|^2\right)^{\frac{q}{2}}{\rm d}A_{\alpha}(z)\\
&\simeq\int_{\mathbb{D}}\int_0^1\left|\sum_{j=1}^{\infty}c_jr_j(t)\left(u(z)\frac{(1-|a_j|^2)^{N-\frac{2+\alpha}{p}}}{(1-\overline{a}_j\varphi(z))^N}-v(z)\frac{(1-|a_j|^2)^{N-\frac{2+\alpha}{p}}}{(1-\overline{a}_j\psi(z))^N}\right)\right|^q{\rm d}t{\rm d}A_{\alpha}(z)\\
&=\int_{0}^1\int_{\mathbb{D}}\left|\sum_{j=1}^{\infty}c_jr_j(t)\left(u(z)\frac{(1-|a_j|^2)^{N-\frac{2+\alpha}{p}}}{(1-\overline{a}_j\varphi(z))^N}-v(z)\frac{(1-|a_j|^2)^{N-\frac{2+\alpha}{p}}}{(1-\overline{a}_j\psi(z))^N}\right)\right|^q{\rm d}A_{\alpha}(z){\rm d}t\\
&\lesssim \|\{c_j\}\|_{l^p}^q.
\end{split}
\end{equation*}
By \eqref{equa2.3}, $\chi_{D(a_j,2r)}(\varphi(z))\lesssim \frac{1-|a_j|^2}{|1-\overline{a}_j\varphi(z)|}$. It follows that
\begin{equation*}
\begin{split}
\int_{\mathbb{D}}\left(\sum_{j=1}^{\infty}\frac{|c_j|^2\chi_{D(a_j,2r)}(\varphi(z))\left|u(z)-v(z)\left(\frac{1-\overline{a}_j\varphi(z)}{1-\overline{a}_j\psi(z)}\right)^N\right|^2}{(1-|a_j|^2)^{2(2+\alpha)/p}}\right)^{q/2}{\rm d}A_{\alpha(z)} \lesssim \|\{c_j\}\|_{l^p}^q.
\end{split}
\end{equation*}
Recall that there is a positive integer $N_0$ such that each point $z\in\mathbb{D}$ belongs to at most $N_0$ of the disks $D(a_j,2r)$. Applying Minkowski's inequality if $\frac{2}{q}\leq 1$ and H\"older's inequality if $\frac{2}{q}>1$, we obtain
\begin{equation*}
\begin{split}
&\sum_{j=1}^{\infty}|c_j|^q\frac{\int_{\varphi^{-1}(D(a_j,2r))}\left|u(z)-v(z)\left(\frac{1-\overline{a}_j\varphi(z)}{1-\overline{a}_j\psi(z)}\right)^N\right|^q{\rm d}A_{\alpha}(z)}{(1-|a_j|^2)^{q(2+\alpha)/p}}\\
&\quad\leq C\int_{\mathbb{D}}\left(\sum_{j=1}^{\infty}\frac{|c_j|^2\chi_{D(a_j,2r)}(\varphi(z))\left|u(z)-v(z)\left(\frac{1-\overline{a}\varphi(z)}{1-\overline{a}\psi(z)}\right)^N\right|^2}{(1-|a_j|^2)^{2(2+\alpha)/p}}\right)^{q/2}{\rm d}A_{\alpha(z)}\\
&\quad \lesssim \|\{c_j\}\|_{l^p}^q,
\end{split}
\end{equation*}
where $C=\max\{N_0^{1-\frac{q}{2}},1\}$. This implies that the sequence 
\[\left\{\frac{\int_{\varphi^{-1}(D(a_j,2r))}\left|u(z)-v(z)\left(\frac{1-\overline{a}_j\varphi(z)}{1-\overline{a}_j\psi(z)}\right)^N\right|^q{\rm d}A_{\alpha}(z)}{(1-|a_j|^2)^{q(2+\alpha)/p}}\right\}_{j=1}^{\infty}\]
belongs to the dual of $l^{\frac{p}{q}}$ or, equivalently,
\begin{equation}\label{equa3.7}
\sum_{j=1}^{\infty}\left(\frac{\int_{\varphi^{-1}(D(a_j,2r))}\left|u(z)-v(z)\left(\frac{1-\overline{a}_j\varphi(z)}{1-\overline{a}_j\psi(z)}\right)^N\right|^q{\rm d}A_{\alpha}(z)}{(1-|a_j|^2)^{q(2+\alpha)/p}}\right)^{\frac{p}{p-q}}<\infty.
\end{equation}
Similarly, taking $i=1$, we obtain
\begin{equation}\label{equa3.8}
\sum_{j=1}^{\infty}\left(\frac{\int_{\varphi^{-1}(D(a_j,2r))}\left|u(z)\varphi(z)-v(z)\psi(z)\left(\frac{1-\overline{a}_j\varphi(z)}{1-\overline{a}_j\psi(z)}\right)^{N+1}\right|^qdA_{\alpha}(z)}{(1-|a_j|^2)^{q(2+\alpha)/p}}\right)^{\frac{p}{p-q}}<\infty.
\end{equation}
By \eqref{equa2.3}, $\left|\frac{1-\overline{a}_j\varphi(z)}{1-\overline{a}_j\psi(z)}\right|\lesssim 1$ when $\varphi(z)\in D(a,r)$ and constant suppressed is independent of $a_j$, multiplying the integrand in \eqref{equa3.7} by $|\psi(z)|^q\left|\frac{1-\overline{a_j}\varphi(z)}{1-\overline{a}_j\psi(z)}\right|^q$, adding it to \eqref{equa3.8} and by the triangle inequality and \eqref{equa2.4}, we get
\begin{equation}\label{equa3.9}
\sum_{j=1}^{\infty}\left(\frac{\int_{\varphi^{-1}(D(a_j,2r))}|u(z)|^q\rho(z)^q{\rm d}A_{\alpha}(z)}{(1-|a_j|^2)^{q(2+\alpha)/p}}\right)^{\frac{p}{p-q}}<\infty.
\end{equation}
Therefore, by \eqref{equa2.2} and \eqref{equa2.3}, we obtain
\begin{equation}\label{equa3.10}
\begin{split}
&\int_{\mathbb{D}}\left(\frac{\int_{\varphi^{-1}(D(w,r))}|u(z)|^q\rho(z)^q{\rm d}A_{\alpha}(z)}{(1-|w|^2)^{2+\alpha}}\right)^{\frac{p}{p-q}}{\rm d}A_{\alpha}(w)\\
&\quad \leq \sum_{j=1}^{\infty}\int_{D(a_j,r)}\left(\frac{\int_{\varphi^{-1}(D(w,r))}|u(z)|^q\rho(z)^q{\rm d}A_{\alpha}(z)}{(1-|w|^2)^{2+\alpha}}\right)^{\frac{p}{p-q}}{\rm d}A_{\alpha}(w)\\
&\quad \lesssim \sum_{j=1}^{\infty}\left(\frac{\int_{\varphi^{-1}(D(a_j,2r))}|u(z)|^q\rho(z)^q{\rm d}A_{\alpha}(z)}{(1-|a_j|^2)^{2+\alpha}}\right)^{\frac{p}{p-q}}A_{\alpha}(D(a_j,2r))\\
&\quad \lesssim \sum_{j=1}^{\infty}\left(\frac{\int_{\varphi^{-1}(D(a_j,2r))}|u(z)|^q\rho(z)^q{\rm d}A_{\alpha}(z)}{(1-|a_j|^2)^{q(2+\alpha)/p}}\right)^{\frac{p}{p-q}}<\infty.
\end{split}
\end{equation}
This implies that $\omega_{\varphi,u}^q$ is a $(p,q)$-Bergman Carleson measure by Lemma \ref{lemma2.7}. Similarly, $\omega_{\psi,v}^q$ is also a $(p,q)$-Bergman Carleson measure.

On the other hand, when $\varphi(z)\in D(a_j,2r)$, by \eqref{equa3.4}, \eqref{equa3.5} and \eqref{equa2.4}, we have
\begin{equation*}
\begin{split}
|u(z)-v(z)|(1-\rho(z))^N&\lesssim |u(z)-v(z)|\left|\frac{1-\overline{a}_j\varphi(z)}{1-\overline{a}_j\psi(z)}\right|^N\\
&\lesssim \left|u(z)-v(z)\left(\frac{1-\overline{a}_j\varphi(z)}{1-\overline{a}_j\psi(z)}\right)^N\right|+|u(z)|\rho(z).
\end{split}
\end{equation*}
So we combine \eqref{equa3.7} and \eqref{equa3.9} to obtain
\begin{equation*}
\sum_{j=1}^{\infty}\left(\frac{\int_{\varphi^{-1}(D(a_j,2r))}|u(z)-v(z)|^q(1-\rho(z))^{qN}{\rm d}A_{\alpha}(z)}{(1-|a_j|^2)^{q(2+\alpha)/p}}\right)^{\frac{p}{p-q}}<\infty.
\end{equation*}
Then through a similar argument as \eqref{equa3.10}, we get
\begin{equation*}
\int_{\mathbb{D}}\left(\frac{\int_{\varphi^{-1}(D(w,r))}|u(z)-v(z)|^q(1-\rho(z))^{qN}{\rm d}A_{\alpha}(z)}{(1-|w|^2)^{2+\alpha}}\right)^{\frac{p}{p-q}}{\rm d}A_{\alpha}(w)<\infty.
\end{equation*}
This implies that $\sigma_{\varphi,\beta}^q$ is a $(p,q)$-Bergman Carleson measure for $\beta\geq qN>\frac{q}{p}(\alpha+1+\max\{1,p\})$ by Lemma \ref{lemma2.7}. Similarly, $\sigma_{\psi,\beta}^q$ is also a $(p,q)$-Bergman Carleson measure. The proof is now complete.
\end{proof}

According to the above proof, we conclude the following corollary.

\begin{corollary}\label{corollary3.2}
Let $\alpha>-1$, $0<q<p<\infty$ and $N>\max\{1,\frac{1}{p}\}+\frac{1+\alpha}{p}$. Suppose $u,v\in H(\mathbb{D})$ and $\varphi,\psi\in S(\mathbb{D})$. Then $C_{u,\varphi}-C_{v,\psi}: A_{\alpha}^p(\mathbb{D})\to A_{\alpha}^q(\mathbb{D})$ is bounded if and only it is compact, if and only if
\[\|(C_{u,\varphi}-C_{v,\psi})H_{{\mathbf{c}},N,p}^{[i]}\|_{A_{\alpha}^q}\lesssim \|\{c_j\}\|_{l^p},\quad \forall i\in\mathbb{N}\]
for any $\mathbf{c}=\{c_j\}\in l^p$.
\end{corollary}

\section{Proof of Theorem 1.3}

In this section, we aim to prove Theorem \ref{theorem1.3}, which characterizes the Schatten class difference of weighted composition operators on $A_{\alpha}^2(\mathbb{D})$. And then we obtain some direct corollaries.

\begin{proof}[{\bf Proof of Theorem 1.3}]
{\it Sufficiency}: Fix $r>0$, set 
$$G_r=\{z\in\mathbb{D}:\beta(\varphi(z),\psi(z))<r\}.$$
For any $f\in A_{\alpha}^2(\mathbb{D})$, we have
\begin{equation*}
\begin{split}
\langle (C_{u,\varphi}-C_{v,\psi})^*(C_{u,\varphi}-C_{v,\psi})f,f\rangle&=\|(C_{u,\varphi}-C_{v,\psi})f\|_{A_{\alpha}^2}^2\\
&=\int_{\mathbb{D}\backslash G_r}+\int_{G_r}|u\cdot f\circ\varphi-v\cdot f\circ\psi|^2{\rm d}A_{\alpha}\\
&:=I(f)+II(f).
\end{split}
\end{equation*}
Now we estimate each integral separately. For the integral $I(f)$, note that $\rho(z)\geq \tanh(r)$ for $z\in \mathbb{D}\backslash G_r$. So by \eqref{equa2.7}, we obtain
\begin{equation}\label{equa4.1}
\begin{split}
I(f)&\lesssim \int_{\mathbb{D}\backslash G_r}\left(|u\cdot f\circ\varphi|^2+|v\cdot f\circ\psi|^2\right){\rm d}A_{\alpha}\\
&\lesssim \int_{\mathbb{D}}\left(|u\rho|^2|f\circ\varphi|^2+|v\rho|^2|f\circ\psi|^2\right){\rm d}A_{\alpha}\\
&\lesssim \int_{\mathbb{D}}|f|^2M_{2r}(\omega_2){\rm d}A_{\alpha}.
\end{split}
\end{equation}
Meanwhile, clearly,
\begin{align*}
II(f)&\lesssim \int_{G_r}|u-v|^2(|f\circ\varphi|^2+|f\circ\psi|^2)+(|u|^2+|v|^2)|f\circ\varphi-f\circ\psi|^2{\rm d}A_{\alpha}\\
&:=II_1(f)+II_2(f).
\end{align*}
For the integral $II_1(f)$, note that $1-\rho(z)\geq 1-\tanh(r)$ for $z\in G_r$, so by \eqref{equa2.7}, we obtain
\begin{equation}\label{equa4.2}
\begin{split}
II_1(f)&\lesssim \int_{G_r}|u-v|^2(1-\rho)^{\beta}(|f\circ\varphi|^2+|f\circ\psi|^2){\rm d}A_{\alpha}\\
&\lesssim \int_{\mathbb{D}}|f|^2M_{2r}(\sigma_{\beta,2}){\rm d}A_{\alpha}.
\end{split}
\end{equation}
For the integral $II_2(f)$, applying Lemma \ref{lemma2.1}, when $z\in G_r$, we have
\begin{equation*}
|f(\varphi(z))-f(\psi(z))|^2\lesssim \frac{\rho(z)^2}{(1-|\varphi(z)|^2)^{2+\alpha}}\int_{D(\varphi(z),2r)}|f(w)|^2{\rm d}A_{\alpha}(w),
\end{equation*}
and
\begin{equation*}
|f(\varphi(z))-f(\psi(z))|^2\lesssim \frac{\rho(z)^2}{(1-|\psi(z)|^2)^{2+\alpha}}\int_{D(\psi(z),2r)}|f(w)|^2{\rm d}A_{\alpha}(w).
\end{equation*}
Then we use Fubini's Theorem to obtain
\begin{equation}\label{equa4.3}
\begin{split}
II_2(f)&\lesssim\int_{G_r}\frac{|u(z)|^2\rho(z)^2}{(1-|\varphi(z)|^2)^{2+\alpha}}\left(\int_{D(\varphi(z),2r)}|f(w)|^2{\rm d}A_{\alpha}(w)\right){\rm d}A_{\alpha}(z)\\
&\quad +\int_{G_r}\frac{|v(z)|^2\rho(z)^2}{(1-|\psi(z)|^2)^{2+\alpha}}\left(\int_{D(\psi(z),2r)}|f(w)|^2{\rm d}A_{\alpha}(w)\right){\rm d}A_{\alpha}(z)\\
&\lesssim \int_{\mathbb{D}}|f|^2M_{2r}(\omega_2){\rm d}A_{\alpha}.
\end{split}
\end{equation}

Let ${\rm d}\mu=(M_{2r}(\omega_2)+M_{2r}(\sigma_{\beta,2})){\rm d}A_{\alpha}$. Combining \eqref{equa4.1}, \eqref{equa4.2} and \eqref{equa4.3}, we obtain
\[(C_{u,\varphi}-C_{v,\psi})^*(C_{u,\varphi}-C_{v,\psi})\lesssim T_{\mu}.\]
By Lemma \ref{lemma2.8}, the assumption implies that $T_{\mu}\in S_{\frac{p}{2}}(A_{\alpha}^2)$. It follows that $(C_{u,\varphi}-C_{v,\psi})^*(C_{u,\varphi}-C_{v,\psi})\in S_{\frac{p}{2}}(A_{\alpha}^2)$, and then $C_{u,\varphi-}C_{v,\psi}\in S_p(A_{\alpha}^2)$.

{\it Necessity}: Assume $C_{u,\varphi}-C_{v,\psi}\in S_p(A_{\alpha}^2)$ and $p\geq 2$. Then $(C_{u,\varphi}-C_{v,\psi})^*(C_{u,\varphi}-C_{v,\psi})\in S_{\frac{p}{2}}(A_{\alpha}^2)$. And it follows from Lemma \ref{lemma2.6} that the function
\[z\mapsto \|(C_{u,\varphi}-C_{v,\psi})k_z^{[n]}\|_{A_{\alpha}^2}^2\]
belongs to $L^{\frac{p}{2}}(\mathbb{D},{\rm d}\lambda)$.

For any $r>0$, by \eqref{equa2.3} and \eqref{equa2.6}, we have
\begin{equation}\label{equa4.4}
\begin{split}
&\|(C_{u,\varphi}-C_{v,\psi})k_z^{[0]}\|_{A_{\alpha}^2}^2\\
&\gtrsim\int_{\varphi^{-1}(D(z,r))}\left|\frac{u(w)}{(1-\overline{z}\varphi(w))^{2+\alpha}}-
\frac{v(w)}{(1-\overline{z}\psi(w))^{2+\alpha}}\right|^2(1-|z|^2)^{2+\alpha}{\rm d}A_{\alpha}(w)\\
&\gtrsim \frac{\int_{\varphi^{-1}(D(z,r))}\left|u(w)-v(w)\frac{(1-\overline{z}\varphi(w))^{2+\alpha}}{(1-\overline{z}\psi(w))^{2+\alpha}}\right|^2{\rm d}A_{\alpha}(w)}{(1-|z|^2)^{2+\alpha}}.
\end{split}
\end{equation}
Similarly,
\begin{equation}\label{equa4.5}
\begin{split}
&\|(C_{u,\varphi}-C_{v,\psi})k_z^{[1]}\|_{A_{\alpha}^2}^2\\
&\gtrsim\frac{\int_{\varphi^{-1}(D(z,r))}\left|u(w)\varphi(w)-v(w)\psi(w)\frac{(1-\overline{z}\varphi(w))^{3+\alpha}}{(1-\overline{z}\psi(w))^{3+\alpha}}\right|^2{\rm d}A_{\alpha}(w)}{(1-|z|^2)^{2+\alpha}}.
\end{split}
\end{equation}
Since $\left|\frac{1-\overline{z}\varphi(w)}{1-\overline{z}\psi(w)}\right|\lesssim 1$ when $w\in \varphi^{-1}(D(z,r))$, multiplying the integrand in \eqref{equa4.4} by $|\psi(w)|^2\left|\frac{1-\overline{z}\varphi(w)}{1-\overline{z}\psi(w)}\right|^2$, and adding it to \eqref{equa4.5}, by the triangle inequality and \eqref{equa2.4}, we obtain
\begin{equation}\label{equa4.6}
\begin{split}
&\|(C_{u,\varphi}-C_{v,\psi})k_z^{[0]}\|_{A_{\alpha}^2}^2+\|(C_{u,\varphi}-C_{v,\psi})k_z^{[1]}\|_{A_{\alpha}^2}^2\\
&\quad \gtrsim\frac{\int_{\varphi^{-1}(D(z,r))}|u(w)|^2\rho(w)^2{\rm d}A_{\alpha}(w)}{(1-|z|^2)^{2+\alpha}}=M_r(\omega_{\varphi,u}^2)(z).
\end{split}
\end{equation}
Thus $M_r(\omega_{\varphi,u}^2)\in L^{\frac{p}{2}}(\mathbb{D},{\rm d}\lambda)$. Similarly, $M_r(\omega_{\psi,v}^2)\in L^{\frac{p}{2}}(\mathbb{D},{\rm d}\lambda)$.

On the other hand, when $\varphi(w)\in D(z,r)$, by \eqref{equa2.4}, we have
\begin{equation*}
\begin{split}
&\left|1-\frac{(1-\overline{z}\varphi(w))^{\alpha+2}}{(1-\overline{z}\psi(w))^{\alpha+2}}\right|\\
&\quad\lesssim \sup_{0\leq t\leq 1}\frac{|1-\overline{z}(t\varphi(w)+(1-t)\psi(w))|^{\alpha+1}}{|1-\overline{z}\psi(w)|^{\alpha+1}}\frac{|\varphi(w)-\psi(w)|}{|1-\overline{z}\psi(w)|}\lesssim \rho(z,w).
\end{split}
\end{equation*}
It follows that when $\varphi(w)\in D(z,r)$, 
\begin{align*}
&|u(w)-v(w)|(1-\rho(w))^{2+\alpha}\\
&\quad\lesssim |u(w)-v(w)|\left|\frac{1-\overline{z}\varphi(w)}{1-\overline{z}\psi(w)}\right|^{\alpha+2}\\
&\quad \leq \left|u(w)-v(w)\frac{(1-\overline{z}\varphi(w))^{\alpha+2}}{(1-\overline{z}\psi(w))^{\alpha+2}}\right|+|u(w)|\left|1-\frac{(1-\overline{z}\varphi(w))^{\alpha+2}}{(1-\overline{z}\psi(w))^{\alpha+2}}\right|\\
&\quad \lesssim \left|u(w)-v(w)\frac{(1-\overline{z}\varphi(w))^{\alpha+2}}{(1-\overline{z}\psi(w))^{\alpha+2}}\right|+|u(w)|\rho(w).
\end{align*}
So when $\beta\geq 2(\alpha+2)$ we combine \eqref{equa4.4} and \eqref{equa4.6} to obtain
\begin{equation*}
\begin{split}
&\|(C_{u,\varphi}-C_{v,\psi})k_z^{[0]}\|_{A_{\alpha}^2}^2+\|(C_{u,\varphi}-C_{v,\psi})k_z^{[1]}\|_{A_{\alpha}^2}^2\\
&\quad \gtrsim\frac{\int_{\varphi^{-1}(D(z,r))}|u(w)-v(w)|^2(1-\rho(w))^{2(2+\alpha)}{\rm d}A_{\alpha}(w)}{(1-|z|^2)^{2+\alpha}}\geq M_r(\sigma_{\varphi,\beta}^2)(z).
\end{split}
\end{equation*}
Consequently, $M_r(\sigma_{\varphi,\beta}^2)\in L^{\frac{p}{2}}(\mathbb{D},{\rm d}\lambda)$. Similarly, $M_r(\sigma_{\psi,\beta}^2)\in L^{\frac{p}{2}}(\mathbb{D},{\rm d}\lambda)$. The proof is complete.
\end{proof}

\begin{remark}\label{remark4.1}
According to Theorem \ref{theorem1.3}, $C_{u,\varphi}-C_{v,\psi}$ is Hilbert-Schmidt on $A_{\alpha}^2(\mathbb{D})$ if and only if
\begin{itemize}
\item[(a)] $M_r(\omega_{2})$ and $M_r(\sigma_{\beta,2})$ belong to $L^1(\mathbb{D},{\rm d}\lambda)$ for some (or any) $r>0$. Here, $\beta=2(2+\alpha)$.
\end{itemize}
This is equivalent to \cite[Theorem 2]{1}, i.e. $C_{u,\varphi}-C_{v,\psi}$ is Hilbert-Schmidt on $A_{\alpha}^2(\mathbb{D})$ if and only if 
\begin{itemize}
\item[(b)] the functions $\frac{|\rho u|^2}{(1-|\varphi|^2)^{2+\alpha}}$, $\frac{|\rho v|^2}{(1-|\psi|^2)^{2+\alpha}}$, $\frac{(1-\rho)^{2+\alpha}|u-v|^2}{(1-|\varphi|^2)^{2+\alpha}}$ and $\frac{(1-\rho)^{2+\alpha}|u-v|^2}{(1-|\psi|^2)^{2+\alpha}}$ belong to $L^1(\mathbb{D},{\rm d}A_{\alpha})$.
\end{itemize}

In fact, by \eqref{equa2.3} and Fubini's Theorem, 
\begin{equation}\label{equa4.7}
\begin{split}
&\int_{\mathbb{D}}M_r(\omega_2)(z){\rm d}\lambda(z)\\
&\quad =\int_{\mathbb{D}}\frac{\int_{\varphi^{-1}(D(z,r))}|u(w)|^2\rho(w)^2{\rm d}A_{\alpha}(w)}{(1-|z|^2)^{2+\alpha}}{\rm d}\lambda(z)\\
&\qquad \quad +\int_{\mathbb{D}}\frac{\int_{\psi^{-1}(D(z,r))}|v(w)|^2\rho(w)^2{\rm d}A_{\alpha}(w)}{(1-|z|^2)^{2+\alpha}}{\rm d}\lambda(z)\\
&\quad =\int_{\mathbb{D}}|u(w)|^2\rho(w)^2\left(\int_{D(\varphi(w),r)}\frac{1}{(1-|z|^2)^{2+\alpha}}{\rm d}\lambda(z)\right){\rm d}A_{\alpha}(w)\\
&\qquad \quad +\int_{\mathbb{D}}|v(w)|^2\rho(w)^2\left(\int_{D(\psi(w),r)}\frac{1}{(1-|z|^2)^{2+\alpha}}{\rm d}\lambda(z)\right){\rm d}A_{\alpha}(w)\\
&\quad \simeq \int_{\mathbb{D}}\frac{|u(w)|^2\rho(w)^2}{(1-|\varphi(w)|^2)^{2+\alpha}}+\frac{|v(w)|^2\rho(w)^2}{(1-|\psi(w)|^2)^{2+\alpha}}{\rm d}A_{\alpha}(w)
\end{split}
\end{equation}
Meanwhile, for $\beta=2(\alpha+2)$,
\begin{equation}\label{equa4.8}
\begin{split}
&\int_{\mathbb{D}}M_r(\sigma_{\beta,2}){\rm d}\lambda\\
&\quad \simeq\int_{\mathbb{D}}|u-v|^2(1-\rho)^{2(2+\alpha)}\left(\frac{1}{(1-|\varphi|^2)^{2+\alpha}}+\frac{1}{(1-|\psi|^2)^{2+\alpha}}\right){\rm d}A_{\alpha}\\
&\quad \leq\int_{\mathbb{D}}|u-v|^2(1-\rho)^{2+\alpha}\left(\frac{1}{(1-|\varphi|^2)^{2+\alpha}}+\frac{1}{(1-|\psi|^2)^{2+\alpha}}\right){\rm d}A_{\alpha}.
\end{split}
\end{equation}
On the other hand, for fixed $0<\delta<1$, let $E_{\delta}=\{z:\rho(z)<\delta\}$. We write
\begin{equation*}
\begin{split}
&\int_{\mathbb{D}}|u-v|^2(1-\rho)^{2+\alpha}\left(\frac{1}{(1-|\varphi|^2)^{2+\alpha}}+\frac{1}{(1-|\psi|^2)^{2+\alpha}}\right){\rm d}A_{\alpha}\\
&\quad =\int_{E_{\delta}}+\int_{\mathbb{D}\backslash E_{\delta}}|u-v|^2(1-\rho)^{2+\alpha}\left(\frac{1}{(1-|\varphi|^2)^{2+\alpha}}+\frac{1}{(1-|\psi|^2)^{2+\alpha}}\right){\rm d}A_{\alpha}.
\end{split}
\end{equation*}
Clearly, 
\begin{equation}\label{equa4.9}
\begin{split}
\int_{E_{\delta}}|u-v|^2(1-\rho)^{2+\alpha}\left(\frac{1}{(1-|\varphi|^2)^{2+\alpha}}+\frac{1}{(1-|\psi|^2)^{2+\alpha}}\right){\rm d}A_{\alpha} \lesssim \int_{\mathbb{D}}M_{r}(\sigma_{\beta,2}){\rm d}\lambda.
\end{split}
\end{equation}
And by \eqref{equa2.1}, we have
\begin{equation}\label{equa4.10}
\begin{split}
&\int_{\mathbb{D}\backslash E_{\delta}}|u-v|^2(1-\rho)^{2+\alpha}\left(\frac{1}{(1-|\varphi|^2)^{2+\alpha}}+\frac{1}{(1-|\psi|^2)^{2+\alpha}}\right){\rm d}A_{\alpha}\\
&\quad =\int_{\mathbb{D}\backslash E_{\delta}}|u-v|^2\frac{(1-|\varphi|^2)^{2+\alpha}+(1-|\psi|^2)^{2+\alpha}}{|1-\overline{\varphi}\psi|^{2(2+\alpha)}}{\rm d}A_{\alpha}\\
&\quad \lesssim \int_{\mathbb{D}\backslash E_{\delta}}\frac{(|u|^2+|v|^2)\rho^2}{|1-\overline{\varphi}\psi|^{2+\alpha}}{\rm d}A_{\alpha}\lesssim \int_{\mathbb{D}}M_r(\omega_2){\rm d}\lambda.
\end{split}
\end{equation}
Combining \eqref{equa4.7}-\eqref{equa4.10}, the equivalence of (a) and (b) holds.
\end{remark}

From the proof of Theorem \ref{theorem1.2}, we conclude the following characterization for Schatten class difference $C_{u,\varphi}-C_{v,\psi}$ via Reproducing Kernel Thesis.

\begin{corollary}\label{corollary4.2}
Let $\alpha>-1$ and $p\geq 2$. Suppose $u,v\in H(\mathbb{D})$ and $\varphi,\psi\in S(\mathbb{D})$. Then $C_{u,\varphi}-C_{v,\psi}\in S_p(A_{\alpha}^2)$ if and only if 
\[\int_{\mathbb{D}}\|(C_{u,\varphi}-C_{v,\psi})k_z^{[n]}\|_{A_{\alpha}^2}^p{\rm d}\lambda(z)<\infty\]
for any $n\in\mathbb{N}$. 
\end{corollary}

\begin{remark}\label{remark3.3}
Let $v(z)=0$, when $p\geq 2$, by Theorem \ref{theorem1.3} and Corollary \ref{corollary4.2}, it is easy to see that $C_{u,\varphi}\in S_p(A_{\alpha}^2)$ if and only if
\begin{equation}\label{equa4.11}
\int_{\mathbb{D}}\left(\frac{\int_{\varphi^{-1}(D(z,r))}|u(w)|^2{\rm d}A_{\alpha}(w)}{(1-|z|^2)^{2+\alpha}}\right)^{\frac{p}{2}}{\rm d}\lambda(z)<\infty,
\end{equation}
if and only if
\begin{equation*}
\int_{\mathbb{D}}\|C_{u,\varphi}k_z^{[n]}\|_{A_{\alpha}^2}^p{\rm d}\lambda(z)<\infty,\quad \forall n\in\mathbb{N}.
\end{equation*}
In fact, an easy application of \eqref{equa2.8} shows that \eqref{equa4.11} holds for all $0<p<\infty$, since $C_{u,\varphi}^*C_{u,\varphi}=T_{\mu_{u,\varphi}}$, where $\mu_{u,\varphi}=(|u|^2dA_{\alpha})\circ\varphi^{-1}$.
\end{remark}

\begin{proposition}\label{proposition4.4}
Let $\alpha>-1$, $p\geq 2$ and $a,b\in\mathbb{C}\backslash\{0\}$. Suppose $C_{\varphi}$ and $C_{\psi}$ are not in $S_{p}(A_{\alpha}^2)$. Then $aC_{\varphi}+bC_{\psi}\in S_p(A_{\alpha}^2)$ if and only if $a+b=0$ and $C_{\varphi}-C_{\psi}\in S_p(A_{\alpha}^2)$.
\end{proposition}

\begin{proof}
The sufficiency is trivial. We only need to prove the necessity.

Assume $aC_{\varphi}+bC_{\psi}\in S_p(A_{\alpha}^2)$. Fix $r>0$. By Theorem \ref{theorem1.3}, we have
\begin{equation}\label{equa4.12}
\int_{\mathbb{D}}\left(\frac{\int_{\varphi^{-1}(D(z,r))}\rho(w)^2{\rm d}A_{\alpha}(w)}{(1-|z|^2)^{2+\alpha}}\right)^{\frac{p}{2}}{\rm d}\lambda(z)<\infty,
\end{equation} 
and
\begin{equation}\label{equa4.13}
\int_{\mathbb{D}}\left(\frac{\int_{\varphi^{-1}(D(z,r))}(1-\rho(w))^{2(2+\alpha)}|a+b|^2{\rm d}A_{\alpha}(w)}{(1-|z|^2)^{2+\alpha}}\right)^{\frac{p}{2}}{\rm d}\lambda(z)<\infty.
\end{equation}
If $a+b\neq 0$, then adding \eqref{equa4.12} and \eqref{equa4.13}, we easily obtain
\[\int_{\mathbb{D}}\left(\frac{A_{\alpha}(\varphi^{-1}(D(z,r)))}{(1-|z|^2)^{\alpha+2}}\right)^{\frac{p}{2}}{\rm d}\lambda(z)<\infty.\]
By \eqref{equa4.11}, this shows that $C_{\varphi}\in S_p(A_{\alpha}^2)$, which is a contradiction. It follows that $a+b=0$ and $C_{\varphi}-C_{\psi}\in S_p(A_{\alpha}^2)$.
\end{proof}

\section{Proof of Theorem 1.4}

In this section, we aim to prove Theoremm \ref{theorem1.4}. To this end, we need the following lemma.

\begin{lemma}\label{lemma5.1}
Let $u\in H(\mathbb{D})$ and $\varphi\in S(\mathbb{D})$. Then $C_{u,\varphi}$ is bounded (compact, resp.) on $H^2(\mathbb{D})$ if and only if both $C_{u',\varphi}: H^2(\mathbb{D})\to A_1^2(\mathbb{D})$ and $C_{u\varphi',\varphi}: A_1^2(\mathbb{D})\to A_{1}^2(\mathbb{D})$ are bounded (compact, resp.).
\end{lemma}

\begin{proof}
By \eqref{equa1.1}, we know that $f\in H^2(\mathbb{D})$ if and only if $f'\in A_1^2(\mathbb{D})$. So the suffciency follows from the fact that 
\begin{align*}
\|(C_{u,\varphi}f)'\|_{A_1^2}^2&=\int_{\mathbb{D}}|u'(z)f(\varphi(z))+u(z)\varphi'(z)f'(\varphi(z))|^2{\rm d}A_1(z)\\
&\lesssim \|C_{u',\varphi}f\|_{A_1^2}^2+\|C_{u\varphi',\varphi}f'\|_{A_{1}^2}^2.
\end{align*}

Now we prove the necessity. Assume $C_{u,\varphi}$ is bounded on $H^2(\mathbb{D})$, we use the test function $g_{a,N,2}^{[i]}$ defined in Section 2.2 to obtain
$$\sup_{a\in\mathbb{D}}\|C_{u,\varphi}g_{a,N,2}^{[i]}\|_{H^2}<\infty,\quad i\in\mathbb{N}.$$
By \eqref{equa1.1}, taking $i=0$, we have
\begin{align*}
&\|C_{u,\varphi}g_{a,N,2}^{[0]}\|_{H^2}^2\\
&\quad \gtrsim \int_{\mathbb{D}}|(C_{u,\varphi}g_{a,N,2}^{[0]})'(z)|^2{\rm d}A_1(z)\\
&\quad =\int_{\mathbb{D}}\left|\frac{u'(z)}{(1-\overline{a}\varphi(z))^N}+\frac{u(z)\varphi'(z)N\overline{a}}{(1-\overline{a}\varphi(z))^{N+1}}\right|^2(1-|a|^2)^{2N-1}{\rm d}A_1(z).
\end{align*}
So we get
\begin{align}\label{equa5.1}
\sup_{a\in\mathbb{D}}\int_{\mathbb{D}}\left|\frac{u'(z)}{(1-\overline{a}\varphi(z))^N}+\frac{u(z)\varphi'(z)N\overline{a}}{(1-\overline{a}\varphi(z))^{N+1}}\right|^2(1-|a|^2)^{2N-1}{\rm d}A_1(z)<\infty.
\end{align}
Similarly, taking $i=1$, we get
\begin{align}\label{equa5.2}
\sup_{a\in\mathbb{D}}\int_{\mathbb{D}}\left|\frac{u'(z)\varphi(z)}{(1-\overline{a}\varphi(z))^{N+1}}+\frac{u(z)\varphi'(z)(1+N\overline{a}\varphi(z))}{(1-\overline{a}\varphi(z))^{N+2}}\right|^2(1-|a|^2)^{2N+1}{\rm d}A_1(z)<\infty.
\end{align}
Note that $\left|\frac{1-|a|^2}{1-\overline{a}\varphi(z)}\right|\lesssim 1$. Multiplying the integrand in \eqref{equa5.1} by $\frac{(1-|a|^2)^2}{|1-\overline{a}\varphi(z)|^2}|\varphi(z)|^2$, adding it to \eqref{equa5.2} and by the triangle inequality, we obtain
\begin{align}\label{equa5.3}
\sup_{a\in\mathbb{D}}\int_{\mathbb{D}}|u(z)\varphi'(z)|^2\frac{(1-|a|^2)^{2N+1}}{|1-\overline{a}\varphi(z)|^{2N+4}}{\rm d}A_1(z)<\infty.
\end{align}
Then it follows from \eqref{equa2.3} that
\begin{align}\label{equa5.4}
\sup_{a\in\mathbb{D}}\frac{\int_{\varphi^{-1}(D(a,r))}|u(z)\varphi'(z)|^2{\rm d}A_1(z)}{(1-|a|^2)^3}<\infty,
\end{align}
which means that $(|u\varphi'|{\rm d}A_1)\circ\varphi^{-1}$ is a Carleson measure for $A_1^2(\mathbb{D})$ by Lemma \ref{lemma2.7}. So $C_{u\varphi',\varphi}$ is bounded on $A_1^2(\mathbb{D})$.

Since \eqref{equa5.1} also holds replacing $N$ by $N+1$, combining \eqref{equa5.1} and \eqref{equa5.3}, we get
\begin{align}\label{equa5.5}
\sup_{a\in\mathbb{D}}\int_{\mathbb{D}}|u'(z)|^2\frac{(1-|a|^2)^{2N+1}}{|1-\overline{a}\varphi(z)|^{2N+2}}{\rm d}A_1(z)<\infty.
\end{align}
This shows that $(|u'|{\rm d}A_1)\circ\varphi^{-1}$ is a Hardy Carleson measure by Equation (1.1) in \cite{25}. It follows that $C_{u',\varphi}: H^2(\mathbb{D})\to A_1^2(\mathbb{D})$ is bounded.

The proof for the compactness part is just a modification. We omit the routine details.
\end{proof}

We are now ready to prove Theorem \ref{theorem1.4}.

\begin{proof}[{\bf Proof of Theorem 1.4}]
The sufficiency is trivial, since for $f\in H^2(\mathbb{D})$, we have $f'\in A_1^2(\mathbb{D})$ and 
\begin{align*}
\left\|\left[(C_{u,\varphi}-C_{v,\psi})f\right]'\right\|_{A_1^2}\lesssim \|(C_{u',\varphi}-C_{v',\psi})f\|_{A_1^2}+\|(C_{u\varphi',\varphi}-C_{v\psi',\psi})f'\|_{A_1^2}.
\end{align*}

Now we prove the necessity. Assume both $C_{u,\varphi}$ and $C_{v,\psi}$ are bounded on $H^2(\mathbb{D})$ and $C_{u,\varphi}-C_{v,\psi}$ is compact on $H^2(\mathbb{D})$. By the triangle inequality, for $f\in H^2(\mathbb{D})$, we have
\begin{align*}
\|(C_{u',\varphi}-C_{v',\psi})f\|_{A_1^2}\lesssim \|(C_{u,\varphi}-C_{v,\psi})f\|_{H^2}+\|(C_{u\varphi',\varphi}-C_{v\psi',\psi})f'\|_{A_1^2}.
\end{align*}
Thus we only need to prove the compactness of $C_{u\varphi',\varphi}-C_{v\psi',\psi}$ on $A_1^2(\mathbb{D})$.

Recall the test function 
$$g_{a,N,2}^{[i]}(z)=\frac{(1-|a|^2)^{N+i-\frac{1}{2}}z^i}{(1-\overline{a}z)^{N+i}},\quad z\in\mathbb{D}.$$
By Lemma \ref{lemma2.4}, the compactness of $C_{u,\varphi}-C_{v,\psi}$ on $H^2(\mathbb{D})$ implies that
\begin{align*}
\lim_{|a|\to 1}\|(C_{u,\varphi}-C_{v,\psi})g_{a,N,2}^{[i]}\|_{H^2}=0,\quad i\in\mathbb{N}.
\end{align*}
By \eqref{equa1.1}, taking $i=0$, we have
\begin{align*}
&\|(C_{u,\varphi}-C_{v,\psi})g_{a,N,2}^{[0]}\|_{H^2}^2\\
&\quad \gtrsim \int_{\mathbb{D}}\left|\left[(C_{u,\varphi}-C_{v,\psi})g_{a,N,2}^{[0]}\right]'(z)\right|^2{\rm d}A_1(z)\\
&\quad =\int_{\mathbb{D}}\left|\frac{u'(z)}{(1-\overline{a}\varphi(z))^{N}}+\frac{u(z)\varphi'(z)N\overline{a}}{(1-\overline{a}\varphi(z))^{N+1}}\right.\\
&\quad\qquad\qquad \left.-\frac{v'(z)}{(1-\overline{a}\psi(z))^{N}}-\frac{v(z)\psi'(z)N\overline{a}}{(1-\overline{a}\psi(z))^{N+1}}\right|^2(1-|a|^2)^{2N-1}{\rm d}A_{1}(z)
\end{align*}
Fix $r>0$. Then by \eqref{equa2.3}, we have
\begin{align*}
&\|(C_{u,\varphi}-C_{v,\psi})g_{a,N,2}^{[0]}\|_{H^2}^2\\
&\quad \gtrsim\frac{\int_{\varphi^{-1}(D(a,r))}\left|\left(u'(z)+\frac{u(z)\varphi'(z)N\overline{a}}{1-\overline{a}\varphi(z)}\right)-\left(v'(z)+\frac{v(z)\psi'(z)N\overline{a}}{1-\overline{a}\psi(z)}\right)\left(\frac{1-\overline{a}\varphi(z)}{1-\overline{a}\psi(z)}\right)^N\right|^2{\rm d}A_1(z)}{(1-|a|^2)}.
\end{align*}
It follows that
\begin{align}\label{equa5.6}
&\lim_{|a|\to 1}\int\limits_{\varphi^{-1}(D(a,r))}\left|\left(v'(z)+\frac{v(z)\psi'(z)N\overline{a}}{1-\overline{a}\psi(z)}\right)\left(\frac{1-\overline{a}\varphi(z)}{1-\overline{a}\psi(z)}\right)^N\right.\nonumber \\
&\qquad\qquad\qquad\qquad\left.-\left(u'(z)+\frac{u(z)\varphi'(z)N\overline{a}}{1-\overline{a}\varphi(z)}\right)\right|^2\frac{{\rm d}A_1(z)}{(1-|a|^2)}=0
\end{align}
Similarly, taking $i=1$ and $i=2$, respectively, we obtain
\begin{align}\label{equa5.7}
&\lim_{|a|\to 1}\int\limits_{\varphi^{-1}(D(a,r))}\left|\left(v'(z)\psi(z)+\frac{v(z)\psi'(z)(1+N\overline{a}\psi(z))}{1-\overline{a}\psi(z)}\right)\left(\frac{1-\overline{a}\varphi(z)}{1-\overline{a}\psi(z)}\right)^{N+1}\right.\nonumber\\
&\qquad\qquad\qquad\left.-\left(u'(z)\varphi(z)+\frac{u(z)\varphi'(z)(1+N\overline{a}\varphi(z))}{1-\overline{a}\varphi(z)}\right)\right|^2\frac{{\rm d}A_1(z)}{1-|a|^2}=0,
\end{align}
and 
\begin{align}\label{equa5.8}
&\lim_{|a|\to 1}\int\limits_{\varphi^{-1}(D(a,r))}\left|\left(v'(z)\psi(z)^2+\frac{v(z)\psi(z)(2\psi(z)+N\overline{a}\psi(z)^2)}{1-\overline{a}\psi(z)}\right)\left(\frac{1-\overline{a}\varphi(z)}{1-\overline{a}\psi(z)}\right)^{N+2}\right.\nonumber\\
&\qquad\qquad\left.-\left(u'(z)\varphi(z)^2+\frac{u(z)\varphi'(z)(2\varphi(z)+N\overline{a}\varphi(z)^2)}{1-\overline{a}\varphi(z)}\right)\right|^2\frac{{\rm d}A_1(z)}{1-|a|^2}=0.
\end{align}

By \eqref{equa2.3}, $\left|\frac{1-\overline{a}\varphi(z)}{1-\overline{a}\psi(z)}\right|\lesssim 1$ when $\varphi(z)\in D(a,r)$. Multiplying the integrand in \eqref{equa5.6} by $\left|\frac{1-\overline{a}\varphi(z)}{1-\overline{a}\psi(z)}\right|^2|\psi(z)|^2$, adding it to \eqref{equa5.7} and by the triangle inequality, we obtain
\begin{align}\label{equa5.9}
&\lim_{|a|\to 1}\int\limits_{\varphi^{-1}(D(a,r))}\left|\frac{v(z)\psi'(z)}{1-\overline{a}\psi(z)}\left(\frac{1-\overline{a}\varphi(z)}{1-\overline{a}\psi(z)}\right)^{N+1}-\frac{u(z)\varphi'(z)}{1-\overline{a}\varphi(z)}\right.\nonumber\\
&\qquad\qquad\left.-\left(u'(z)+\frac{u(z)\varphi'(z)N\overline{a}}{1-\overline{a}\varphi(z)}\right)\frac{\varphi(z)-\psi(z)}{1-\overline{a}\psi(z)}\right|^2\frac{{\rm d}A_1(z)}{1-|a|^2}=0.
\end{align}
Similarly, multiplying the integrand in \eqref{equa5.7} by $\left|\frac{1-\overline{a}\varphi(z)}{1-\overline{a}\psi(z)}\right|^2|\psi(z)|^2$, adding it to \eqref{equa5.8} and by the triangle inequality, we obtain
\begin{align}\label{equa5.10}
&\lim_{|a|\to 1}\int\limits_{\varphi^{-1}(D(a,r))}\left|\frac{v(z)\psi'(z)\psi(z)}{1-\overline{a}\psi(z)}\left(\frac{1-\overline{a}\varphi(z)}{1-\overline{a}\psi(z)}\right)^{N+2}-\frac{u(z)\varphi'(z)\varphi(z)}{1-\overline{a}\varphi(z)}\right.\nonumber\\
&\qquad\qquad\left.-\left(u'(z)\varphi(z)+\frac{u(z)\varphi'(z)(1+N\overline{a}\varphi(z))}{1-\overline{a}\varphi(z)}\right)\frac{\varphi(z)-\psi(z)}{1-\overline{a}\psi(z)}\right|^2\frac{{\rm d}A_1(z)}{1-|a|^2}=0.
\end{align}
Again, multiplying the integrand in \eqref{equa5.9} by $\left|\frac{1-\overline{a}\varphi(z)}{1-\overline{a}\psi(z)}\right|^2|\psi(z)|^2$, adding it to \eqref{equa5.10} and by the triangle inequality, we obtain
\begin{align}\label{equa5.11}
&\lim_{|a|\to 1}\int\limits_{\varphi^{-1}(D(a,r))}\left|2\frac{u(z)\varphi'(z)}{1-\overline{a}\varphi(z)}\frac{\varphi(z)-\psi(z)}{1-\overline{a}\psi(z)}\right.\nonumber\\
&\qquad\qquad\left.+\left(u'(z)+\frac{u(z)\varphi'(z)N\overline{a}}{1-\overline{a}\varphi(z)}\right)\left(\frac{\varphi(z)-\psi(z)}{1-\overline{a}\psi(z)}\right)^2\right|^2\frac{{\rm d}A_1(z)}{1-|a|^2}=0.
\end{align}
Since \eqref{equa5.11} also holds replacing $N$ by $N+1$, we get 
\begin{align*}
\lim_{|a|\to 1}\int_{\varphi^{-1}(D(a,r))}\left|\frac{u(z)\varphi'(z)}{1-\overline{a}\varphi(z)}\right|^2\left|\frac{\varphi(z)-\psi(z)}{1-\overline{a}\psi(z)}\right|^4\frac{{\rm d}A_1(z)}{1-|a|^2}=0.
\end{align*}
Then it follows from \eqref{equa2.3} and \eqref{equa2.4} that
\begin{align}\label{equa5.12}
\lim_{|a|\to 1}\frac{\int_{\varphi^{-1}(D(a,r))}|u(z)\varphi'(z)|^2\rho(z)^4{\rm d}A_1(z)}{(1-|a|^2)^3}=0.
\end{align}
Since $C_{u,\varphi}$ is bounded on $H^2(\mathbb{D})$, combining \eqref{equa5.12} and \eqref{equa5.3}, we obtain
\begin{align}\label{equa5.13}
\lim_{|a|\to 1}\frac{\int_{\varphi^{-1}(D(a,r))}|u(z)\varphi'(z)|^2\rho(z)^2{\rm d}A_1(z)}{(1-|a|^2)^3}=0.
\end{align}
This shows that $(|u\varphi'|^2\rho^2{\rm d}A_1)\circ\varphi^{-1}$ is a vanishing Carleson measure for $A_1^2(\mathbb{D})$. Similarly, $(|v\psi'|^2\rho^2{\rm d}A_1)\circ\psi^{-1}$ is also a vanishing Carleson measure for $A_1^2(\mathbb{D})$.

On the other hand, using \eqref{equa2.3} and \eqref{equa2.4}, combining \eqref{equa5.13} and \eqref{equa5.11}, we obtain
\begin{align}\label{equa5.14}
\lim_{|a|\to 1}\frac{\int_{\varphi^{-1}(D(a,r))}|u'(z)|^2\rho(z)^4{\rm d}A_1(z)}{1-|a|^2}=0.
\end{align}
Since $C_{u,\varphi}$ is bounded on $H^2(\mathbb{D})$, by \eqref{equa5.5} and \eqref{equa2.3}, we have
\begin{align*}
\sup_{a\in\mathbb{D}}\frac{\int_{\varphi^{-1}(D(a,r))}|u'(z)|^2{\rm d}A_1(z)}{1-|a|^2}<\infty.
\end{align*}
This, together with \eqref{equa5.14}, implies that
\begin{align}\label{equa5.15}
\lim_{|a|\to 1}\frac{\int_{\varphi^{-1}(D(a,r))}|u'(z)|^2\rho(z)^2{\rm d}A_1(z)}{1-|a|^2}=0.
\end{align}
Now using \eqref{equa2.3} and \eqref{equa2.4} again, we combine \eqref{equa5.15}, \eqref{equa5.13} and \eqref{equa5.9} to obtain
\begin{align*}
\lim_{|a|\to 1}\frac{\int_{\varphi^{-1}(D(a,r))}\left|v(z)\psi'(z)\left(\frac{1-\overline{a}\varphi(z)}{1-\overline{a}\psi(z)}\right)^{N+2}-u(z)\varphi'(z)\right|^2{\rm d}A_1(z)}{(1-|a|^2)^3}=0.
\end{align*}
Then modifying the proof of $(i)\Rightarrow(ii)$ in Theorem \ref{theorem1.1} (see \eqref{equa3.2}), we obtain $(|u\varphi'-v\psi'|^2(1-\rho)^{2(N+2)}{\rm d}A_1)\circ\varphi^{-1}$ is a vanishing Carleson measure for $A_1^2(\mathbb{D})$. Similarly, $(|u\varphi'-v\psi'|^2(1-\rho)^{2(N+2)}{\rm d}A_1)\circ\psi^{-1}$ is also a vanishing Carleson measure for $A_1^2(\mathbb{D})$. Finally, according to Theorem \ref{theorem1.1}, $C_{u\varphi',\varphi}-C_{v\psi',\psi}$ is compact on $A_1^2(\mathbb{D})$. The proof is complete.
\end{proof}

\begin{theorem}\label{theorem5.2}
Let $a,b\in \mathbb{C}\backslash\{0\}$ and $\varphi,\psi\in S(\mathbb{D})$ be of bounded valence. Suppose $C_{\varphi}$ and $C_{\psi}$ are not compact on $H^2(\mathbb{D})$. Then $aC_{\varphi}+bC_{\psi}$ is compact on $H^2(\mathbb{D})$ if and only if $a+b=0$ and $C_{\varphi}-C_{\psi}$ is compact on $H^2(\mathbb{D})$.
\end{theorem}

\begin{proof}
The sufficiency is trivial. We only need to prove the necessity.

Assume $aC_{\varphi}+bC_{\psi}$ is compact on $H^2(\mathbb{D})$. Then $aC_{\varphi',\varphi}-bC_{\psi',\psi}$ is compact on $A_1^2(\mathbb{D})$. Let $r>0$ and $\beta>3$. By Theorem \ref{theorem1.1}, we have
\begin{align}\label{equa5.16}
\lim_{|w|\to 1}\frac{\int_{\varphi^{-1}(D(w,r))}\rho(z)^2|\varphi'(z)|^2{\rm d}A_1(z)}{(1-|w|^2)^3}=0,
\end{align}
and 
\begin{align}\label{equa5.17}
\lim_{|w|\to 1}\frac{\int_{\varphi^{-1}(D(w,r))}(1-\rho(z))^{\beta}|a\varphi'(z)+b\psi'(z)|^2{\rm d}A_1(z)}{(1-|w|^2)^3}=0.
\end{align}
For $\delta\in (0,1)$, recall that $E_{\delta}=\{z: \rho(z)<\delta\}$. By the proof of \cite[Theorem 1.3]{4}, if $\varphi$ and $\psi$ are of bounded valence, then
\begin{align}\label{equa5.18}
\lim_{|w|\to 1}\frac{\int_{\varphi^{-1}(D(w,r))\cap E_{\delta}}|\varphi'(z)-\psi'(z)|^2{\rm d}A_1(z)}{(1-|w|^2)^3}=0.
\end{align}
Note that $1-\rho(z)\geq 1-\delta$ for $z\in E_\delta$. If $a+b\neq 0$, then by \eqref{equa5.17} and \eqref{equa5.18}, we get
\begin{align*}
\lim_{|w|\to 1}\frac{\int_{\varphi^{-1}(D(w,r))\cap E_{\delta}}|\varphi'(z)|^2{\rm d}A_1(z)}{(1-|w|^2)^3}=0.
\end{align*}
This, together with \eqref{equa5.16}, implies that
\begin{align*}
\lim_{|w|\to 1}\frac{\int_{\varphi^{-1}(D(w,r))}|\varphi'(z)|^2{\rm d}A_1(z)}{(1-|w|^2)^3}=0.
\end{align*}
Then by Lemma \ref{lemma2.7}, $(|\varphi'|^2{\rm d}A_1)\circ\varphi^{-1}$ is a compact Bergman Carleson measure. It follows that $C_{\varphi',\varphi}$ is compact on $A_1^2(\mathbb{D})$. Then $C_{\varphi}$ is compact on $H^2(\mathbb{D})$, which is a contradiction. So we have $a+b=0$ and $C_{\varphi}-C_{\psi}$ is compact on $H^2(\mathbb{D})$.
\end{proof}


\noindent{\bf Acknowledgment}
Tong was supported in part by the National Natural Science Foundation of China (Grant No. 12171136, 12411530045). Yang was supported in part by the National Natural Science Foundation of China (Grant No. 12501103) and the Natural Science Foundation of Hebei Province (Grant No. A2023202031, A2023202037).\\

\noindent{\bf Declaration}
The authors declare that they have no conflicts of interest. No data was used for the research described in this article.




\end{document}